\documentclass[10pt]{amsart}

\usepackage{amssymb,amscd,amsthm, verbatim,amsmath,color,fancyhdr, mathrsfs}
\usepackage{graphicx}
\usepackage{turnstile}

\usepackage[letterpaper, left=2.5cm, right=2.5cm, top=2.5cm,
bottom=2.5cm,dvips]{geometry}

\newtheorem{thm}{Theorem}[section]

\newtheorem{theorem}{Theorem}[section]

\newtheorem{lem}[thm]{Lemma}
\newtheorem{rem}[thm]{Remark}

\newcommand{\rb}{\mathbb{R}}

\newcommand{\pb}{\mathbb{P}}

\def\beqq{\begin{eqnarray*}}
\def\eeqq{\end{eqnarray*}}
\def\be{\begin{enumerate}}
\def\ee{\end{enumerate}}
\def\pn{\hfil\par\noindent}

\let\wh=\widehat

\def\finpr{\hfill \hbox{
\vrule height 1.453ex  width 0.093ex  depth 0ex
\vrule height 1.5ex  width 1.3ex  depth -1.407ex\kern-0.1ex
\vrule height 1.453ex  width 0.093ex  depth 0ex\kern-1.35ex
\vrule height 0.093ex  width 1.3ex  depth 0ex}}
\def\X{\mathcal{X}}

\title[Blow-Up of Modified Navier--Stokes Equations
in Nonhomogeneous Fourier Spaces]{Blow-Up of Modified Navier--Stokes Equations
in Nonhomogeneous Fourier Spaces}
\author{Jamel Benameur}
  \address{Department of Mathematics, College of Science, King Saud University (KSU) Riyadh 11451, Saudi Arabia}
   \email{jbenameur@ksu.edu.sa}
 \author{Lotfi Jlali}
   \address{Department of Mathematics and Statistics, College of Science, Imam Mohammad Ibn Saud Islamic University (IMSIU) Riyadh 11432, Saudi Arabia}
   \email{Lmjlali@imamu.edu.sa}


\date{\today}

\begin{document}
\maketitle
\begin{abstract}
We study the three-dimensional incompressible Navier--Stokes equations with a componentwise
nonlinear damping term of the form $\alpha\sum_{k=1}^3 u_k^{2m+1}e_k$ in Fourier and Fourier--
Gevrey spaces based on $\mathcal{X}^0(\mathbb{R}^3)$. We first establish local-in-time
existence and uniqueness of solutions and obtain the corresponding unique maximal solutions. We
then derive blow-up criteria in the Fourier--Gevrey framework, including an integral blow-up
criterion in a weaker exponential weight. Moreover, we establish a quantitative lower bound for
the solution near a possible finite maximal existence time. By iterating the loss of
exponential weight, we finally obtain a corresponding lower bound in the unweighted Fourier
space $\mathcal{X}^0(\mathbb{R}^3)$.
\end{abstract}
\section{Introduction}

The three-dimensional incompressible Navier--Stokes equations constitute one of the fundamental
models in fluid dynamics. Their mathematical theory has been extensively developed in various
functional frameworks; see, among others, the classical work of Fujita and Kato \cite{FK}.
Various modifications of these equations involving additional dissipative mechanisms have been
introduced and studied in order to investigate their influence on the existence, regularity,
and long-time behavior of solutions \cite{Temam,LLS2021,XZ2024}.
A commonly considered nonlinear damping term is of the form
$$
\alpha |u|^{2m}u,
$$
where $\alpha>0$ and $m\geq1$.

In the present paper, instead of the isotropic damping term $\alpha |u|^{2m}u$, we consider a
componentwise nonlinear damping term of the form
$$
\alpha\sum_{k=1}^3 |u_k|^{2m}u_k e_k = \alpha\sum_{k=1}^3 u_k^{2m+1}e_k,
$$
where $(e_1,e_2,e_3)$ denotes the canonical basis of $\mathbb{R}^3$. Throughout the paper, we
normalize the viscosity coefficient by taking $\nu=1$. More precisely, we study the following
modified Navier--Stokes system:
$$
(NSE_M)
\left\{
\begin{array}{ll}
\partial_t u-\Delta u+u\cdot\nabla u +\alpha\displaystyle\sum_{k=1}^3u_k^{2m+1}e_k=-\nabla p,
&\text{in }\mathbb{R}^+\times\mathbb{R}^3,\\[1mm]
\operatorname{div}u=0,
&\text{in }\mathbb{R}^+\times\mathbb{R}^3,\\[1mm]
u(0,x)=u^0(x),
&\text{in }\mathbb{R}^3.
\end{array}
\right.
$$
Here, $u=(u_1,u_2,u_3)$ denotes the velocity field and $p$ the pressure.

The mathematical analysis of the three-dimensional incompressible Navier--Stokes equations has
been developed in a wide variety of functional settings, with particular attention to spaces
that reflect the scaling and Fourier structure of the equations; see, for instance, \cite{BCD}.
In this direction, Lei and Lin established a global well-posedness result for small initial
data in the critical Fourier space $\mathcal{X}^{-1}(\mathbb{R}^3)$ \cite{LL}. Their approach
motivated several subsequent investigations of the Navier--Stokes equations in Fourier-type
spaces, including critical Fourier--Herz spaces \cite{CW2012}, nonhomogeneous Fourier--Lei--Lin
spaces \cite{Jlali2017}, and, more recently, Lei--Lin--Gevrey spaces \cite{Jlali2025}.  In
particular, these Fourier frameworks provide useful tools for studying well-posedness,
regularity, stability, and the long-time behavior of solutions; see also \cite{IGD}.

Another direction consists in modifying the Navier--Stokes equations by introducing nonlinear
damping terms. Such additional dissipative mechanisms have been considered in several works in
connection with the existence, uniqueness, regularity, and asymptotic behavior of solutions.
In particular, Cai and Jiu \cite{CJ2008} studied the incompressible
Navier--Stokes equations with nonlinear damping and established existence
results for weak and strong solutions. Subsequent investigations addressed
further questions concerning uniqueness and regularity
\cite{ZWL2011,Zhou2012}. These developments show that nonlinear damping
can play an important role in the qualitative analysis of the
three-dimensional system.

In a previous work \cite{BJ2026}, we investigated modified Navier--Stokes equations with
nonlinear damping in Fourier spaces. In that framework, we established local well-posedness and
studied blow-up properties of solutions using Fourier-space estimates and a fixed-point
argument. The purpose of the present work is to extend this analysis to a Fourier--Gevrey
setting. The introduction of an exponential Fourier weight allows us to investigate the
propagation of additional frequency regularity and, at the same time, to obtain refined
information on the behavior of solutions near a possible finite maximal existence time.

To formulate our results, we introduce the Fourier and Fourier--Gevrey spaces that will be used
throughout the paper. Gevrey regularity and analyticity for the Navier--Stokes equations have
been extensively investigated since the pioneering work of Foias and
Temam \cite{FT1989}; see also \cite{Liu1992,BJ2021} and the references
therein. In the Fourier setting, exponential
weights provide a natural way to describe this additional regularity.

Let $a>0$, $\sigma>1$, and $\rho\in\mathbb{R}$. We define
$$
\mathcal{X}^{\rho} :=\left\{u\in\mathcal{S}'(\mathbb{R}^{3})\,;\,\widehat{u}\in
L_{\mathrm{loc}}^1(\mathbb{R}^{3})\ \text{and}\ \int_{\mathbb{R}^{3}}|\xi|^{\rho}|\widehat{u}
(\xi)|\,d\xi<\infty\right\},
$$
and
$$
\mathcal{X}_{a,\sigma}^{\rho}:=\left\{u\in\mathcal{S}'(\mathbb{R}^{3})\,;\,\widehat{u}\in
L_{\mathrm{loc}}^1(\mathbb{R}^{3})\ \text{and}\ \int_{\mathbb{R}^{3}}|\xi|^{\rho}e^{a|\xi|
^{1/\sigma}}|\widehat{u}(\xi)|\,d\xi<\infty\right\}.
$$
The corresponding norms are given by
$$
\|u\|_{\mathcal{X}^{\rho}} =\int_{\mathbb{R}^{3}}|\xi|^{\rho}|\widehat{u}(\xi)|\,d\xi
$$
and
$$
\|u\|_{\mathcal{X}_{a,\sigma}^{\rho}} =\int_{\mathbb{R}^{3}}|\xi|^{\rho} e^{a|\xi|^{1/\sigma}}
|\widehat{u}(\xi)|\,d\xi.
$$
If the initial data $u^0$ are more regular, then, applying the Leray
projector $\mathbb{P}$ to eliminate the pressure term, we obtain the
following Duhamel integral formula:
$$
u(t) = e^{t\Delta}u^0 -\mathcal{B}(u,u)(t)-\alpha\mathcal{D}_m(u)(t),
$$
where
$$
\mathcal{B}(u,v)(t) =\int_0^t e^{(t-\tau)\Delta}\mathbb{P}(u\cdot\nabla v)(\tau)\,d\tau
$$
and
$$
\mathcal{D}_m(u)(t) =\sum_{k=1}^3\int_0^te^{(t-\tau)\Delta}\mathbb{P}\bigl(u_k^{2m+1}e_k\bigr)
(\tau)\,d\tau.
$$
This integral formulation will be used in the analysis developed below.

To further motivate the choice of the Fourier--Gevrey framework, consider the linear system
associated with the Navier--Stokes equations, with initial data $u^0\in\mathcal{X}^0(\mathbb{R}
^3)$. The smoothing properties of the heat semigroup, and their connection with instantaneous
analyticity and Gevrey regularity, are classical; see, for instance, \cite{FT1989,CKV2018} and
the references therein.
$$
\left\{
\begin{array}{ll}
\partial_t v-\Delta v=0,
&\text{in }\mathbb{R}^+\times\mathbb{R}^3,\\
\operatorname{div}v=0,
&\text{in }\mathbb{R}^+\times\mathbb{R}^3,\\
v(0,x)=u^0(x),
&\text{in }\mathbb{R}^3.
\end{array}
\right.
$$
This system admits the unique solution
$$
v(t)=e^{t\Delta}u^0\in\mathcal{C}\bigl(\mathbb{R}^+,\mathcal{X}^0(\mathbb{R}^3)\bigr).
$$
Moreover, this solution enjoys the following regularizing properties:
$$
\begin{array}{ll}
\textnormal{(R1)}
&
v\in L^1(\mathbb{R}^+,\mathcal{X}^2(\mathbb{R}^3)),
\\[2mm]
\textnormal{(R2)}
&
e^{a|D|^r}v(t)\in\mathcal{X}^0(\mathbb{R}^3),
\quad
\forall\,t>0,\ \forall\,a>0,\ 0<r<2,
\\
&
\text{and, for }r=2,\quad 0<a\leq t,
\\[2mm]
\textnormal{(R3)}
&
e^{a|D|^r}v
\in
\mathcal{C}^{\infty}
\bigl((0,\infty),\mathcal{X}^0(\mathbb{R}^3)\bigr),
\quad
\forall\,a>0,\ 0<r<2.
\end{array}
$$
In connection with the Fourier--Gevrey spaces introduced above, we set
$$
r=\frac{1}{\sigma}.
$$
Since $\sigma>1$, we have $0<r<1$.
Our study is therefore restricted to the case
$$
0<r<1,\qquad a>0.
$$
This choice is also motivated by the need to estimate the nonlinear terms through suitable product estimates in exponentially weighted Fourier spaces; see, for instance, \cite{JS2021, BaeBiswas2015}. More precisely, we use
\begin{equation}\label{eqprod1}
\|e^{a|D|^r}(fg)\|_{\mathcal{X}^0}
\leq C
\|e^{a|D|^r}f\|_{\mathcal{X}^0}
\|e^{a|D|^r}g\|_{\mathcal{X}^0}.
\end{equation}
Consequently,
$$
\left\|e^{a|D|^r}(f^{2m+1})\right\|_{\mathcal{X}^0}
\leq
C'\left\|e^{a|D|^r}f\right\|_{\mathcal{X}^0}^{2m+1}.
$$
We now distinguish three cases according to the value of $r$.
\begin{enumerate}
\item[\textbullet] \textbf{Case $r\in(1,2]$.}
For the treatment of the nonlinear terms, and in particular for estimating the product of two
functions as in \eqref{eqprod1}, one is naturally led to consider an inequality of the form
\begin{equation}\label{eqprod2}
e^{a|\xi|^r}\leq C e^{a|\xi-\eta|^r}e^{a|\eta|^r},\qquad\forall\,\xi,\eta\in\mathbb{R}^3.
\end{equation}

For $0<r\leq1$, inequality \eqref{eqprod2} holds; see, for instance, \cite{JS2021}, page 29. This estimate follows from the subadditivity
$$
|\xi|^r \leq |\xi-\eta|^r+|\eta|^r.
$$
However, when $r>1$, this subadditivity property is no longer valid. Instead, one only has
$$
|\xi|^r \leq 2^{r-1}\left(|\xi-\eta|^r+|\eta|^r\right),
$$
which yields
$$
e^{a|\xi|^r} \leq e^{2^{r-1}a|\xi-\eta|^r} e^{2^{r-1}a|\eta|^r}.
$$
Thus, the same exponential parameter $a$ is not preserved in the product estimate.
Consequently, the argument used in the present work does not close in the same Fourier--Gevrey
space when $r>1$.
\item[\textbullet] \textbf{Case $r\in(0,1)$.}
This case corresponds precisely to the Fourier--Gevrey spaces
$$
\mathcal{X}_{a,\sigma}^{0}(\mathbb{R}^3),\qquad \sigma=\frac{1}{r}>1.
$$
Since the function $s\mapsto s^r$ is subadditive for $0<r<1$, we have
$$
|\xi|^r\leq|\xi-\eta|^r+|\eta|^r,\qquad\xi,\eta\in\mathbb{R}^3,
$$
and consequently
$$
e^{a|\xi|^r}\leq e^{a|\xi-\eta|^r}e^{a|\eta|^r}.
$$
This property allows us to control the nonlinear terms in the same Fourier--Gevrey space and
therefore to apply the fixed-point arguments used in the local well-posedness analysis.

Moreover, the strict inequality $0<r<1$ provides an additional property of the exponential weight which will play a crucial role in the analysis of the maximal existence time and in the derivation of the blow-up criteria.

\item[\textbullet] \textbf{Case $r=1$.}
This is the analytic-regularity case. It is also well adapted to our problem and satisfies the
estimates \eqref{eqprod1} and \eqref{eqprod2}. However, the main difficulty lies in proving a
crucial result for our analysis, namely the equality between the maximal existence time in the
Fourier--Gevrey space and that in the Fourier space:
\begin{equation}\label{eqprod3}
T^*\bigl(\mathcal{X}_{a,\sigma}^{0}(\mathbb{R}^3)\bigr)
=
T^*\bigl(\mathcal{X}^{0}(\mathbb{R}^3)\bigr),
\qquad \sigma>1.
\end{equation}
For $0<r<1$, the proof of \eqref{eqprod3} relies on an elementary inequality of the form
(see \cite{JNA2014}, pages 96--97)
\begin{equation}\label{eqprod4}
e^{a|\xi|^r}
\leq
C
e^{a\max\{|\xi-\eta|^r,|\eta|^r\}}
e^{a\theta\min\{|\xi-\eta|^r,|\eta|^r\}},
\qquad
\forall\,\xi,\eta\in\mathbb{R}^3,
\end{equation}
with $\theta\in(0,1)$. This inequality, however, is not valid in the same form when $r=1$.
Consequently, the argument used for $0<r<1$ does not extend directly to the case $r=1$.
\end{enumerate}
The purpose of the present paper is to study the local well-posedness and the behavior near a
possible finite maximal existence time of solutions to the modified Navier--Stokes system $
(NSE_M)$ in the Fourier space $\mathcal{X}^0(\mathbb{R}^3)$ and in the Fourier--Gevrey spaces $
\mathcal{X}_{a,\sigma}^0(\mathbb{R}^3)$, with $a>0$ and $\sigma>1$.

We first establish the local existence and uniqueness of solutions in $\mathcal{X}^0(\mathbb{R}
^3)$ by means of a fixed-point argument. We then derive a continuation criterion for the
corresponding maximal solution. The same analysis is subsequently developed in the Fourier--
Gevrey space $\mathcal{X}_{a,\sigma}^0(\mathbb{R}^3)$.

Our main objective is to obtain refined blow-up criteria in the Fourier--Gevrey framework. More
precisely, assuming that the maximal existence time $T^*$ is finite, we show the divergence of
the Fourier--Gevrey norm as $t$ approaches $T^*$ and derive an integral blow-up criterion in a
weaker exponential weight. We also establish a quantitative lower bound for the solution near
$T^*$. By iterating the loss of exponential weight and passing to the limit, we finally obtain
a corresponding lower bound in the unweighted Fourier space $\mathcal{X}^0(\mathbb{R}^3)$.

We now state the main results of the paper.
\begin{theorem}\label{theo1}
(Local existence and uniqueness in $\mathcal{X}^0(\mathbb{R}^3)$)\\
Let $u^{0}\in \mathcal{X}^{0}(\mathbb{R}^{3})$ be a divergence-free vector field. Then there
exists a time $T=T(u^{0})>0$ such that problem $(NSE_M)$ admits a unique solution
$$
u \in \mathcal{C}_{T}\bigl(\mathcal{X}^{0}(\mathbb{R}^{3})\bigr)\cap L^{1}_{T}\bigl(\mathcal{X}
^{2}(\mathbb{R}^{3})\bigr).
$$
\end{theorem}
\begin{theorem}\label{theo2}(Maximal solution in $\mathcal{X}^0(\mathbb{R}^3)$)\\
Let $u^{0}\in \mathcal{X}^{0}(\mathbb{R}^{3})$ be a divergence-free vector field and let
$$
u\in \mathcal{C}
([0,T^*),\mathcal{X}^{0}(\mathbb{R}^{3}))\cap L^{1}_{loc}([0,T^*),\mathcal{X}^{2}(\mathbb{R}^{3}))
$$
be the maximal solution of problem $(NSE_M)$ given by Theorem \ref{theo1}. Then, for all
$t\in[0,T^*)$, we have
\begin{equation}\label{theo2-eqn1}
\|u(t)\|_{\mathcal{X}^0} \leq \|u^0\|_{\mathcal{X}
^0}\exp\Big(\int_0^t (\frac{1}{2}\|u(z)\|_{\mathcal{X}^0}^2+\alpha \|u(z)\|_{\mathcal{X}
^0}^{2m})dz\Big).
\end{equation}
\end{theorem}
\begin{theorem}\label{theo3}(Local existence and uniqueness in $\mathcal{X}^0_{a,\sigma}(\mathbb{R}^3)$)\\
Let $u^{0}\in \mathcal{X}_{a,\sigma}^{0}(\mathbb{R}^{3})$ be a divergence-free vector field.
Then there exists a time $T=T(u^{0})>0$ such that problem $(NSE_M)$ admits a unique solution
$$
u \in \mathcal{C}_{T}( \mathcal{X}_{a,\sigma}^{0}(\mathbb{R}^{3}))\cap L^{1}_{T}(\mathcal{X}_{a,\sigma}^{2}(\mathbb{R}^{3})).
$$
\end{theorem}
\begin{theorem}\label{theo4}(Blow-up results in $\mathcal{X}_{a,\sigma}^0(\mathbb{R}^3)$)\\
Let $u^{0}\in \mathcal{X}_{a,\sigma}^{0}(\mathbb{R}^{3})$ be a divergence-free vector field and
let
$$
u\in \mathcal{C}([0,T^*),\mathcal{X}_{a,\sigma}^{0}(\mathbb{R}^{3}))\cap L^{1}_{loc}([0,T^*),
\mathcal{X}_{a,\sigma}^{2}(\mathbb{R}^{3}))
$$
be the maximal solution of problem $(NSE_M)$ given by Theorem \ref{theo3}. Then, for all
$t\in[0,T^*)$, we have
\begin{equation}\label{theo4-eqn1}
\|u(t)\|_{\mathcal{X}_{a,\sigma}^0}
\leq
\|u^0\|_{\mathcal{X}_{a,\sigma}^0}
\exp\left(
\int_0^t
\left(
\frac{1}{2}\|u(\tau)\|_{\mathcal{X}_{a,\sigma}^0}^{2}
+
\alpha\|u(\tau)\|_{\mathcal{X}_{a,\sigma}^0}^{2m}
\right)d\tau
\right).
\end{equation}
Moreover, if $T^*<\infty$, then
\begin{equation}\label{theo4-eqn2}
\limsup_{t\nearrow T^*}
\|u(t)\|_{\mathcal{X}_{a,\sigma}^0}
=\infty.
\end{equation}
Furthermore, the following integral blow-up criterion holds:
\begin{equation}\label{theo4-eqn3}
\int_0^{T^*}
\|u(\tau)\|_{\mathcal{X}_{\frac{a}{\sigma},\sigma}^0}^{2m}
\,d\tau
=\infty.
\end{equation}
In addition, for every $t\in[0,T^*)$,
\begin{equation}\label{theo4-eqn4}
\frac{1}{(T^*-t)^{\frac{1}{2m}}}
\leq
(2mC_{m,\alpha})^{\frac{1}{2m}}
e^{C_{m,\alpha}T^*}
\|u(t)\|_{\mathcal{X}_{\frac{a}{\sigma},\sigma}^0}.
\end{equation}
Finally, the corresponding lower bound also holds in the unweighted Fourier space:
\begin{equation}\label{theo4-eqn5}
\frac{1}{(T^*-t)^{\frac{1}{2m}}}
\leq
(2mC_{m,\alpha})^{\frac{1}{2m}}
e^{C_{m,\alpha}T^*}
\|u(t)\|_{\mathcal{X}^0},
\qquad
\forall\,t\in[0,T^*).
\end{equation}
\end{theorem}
\section{\bf Preliminary Results}
In this section, we collect some preliminary estimates in the Fourier and Fourier--Gevrey
spaces that will be used throughout the paper.
\begin{lem}\label{lem1}\pn
Let $f\in\mathcal{X}^{0}(\mathbb{R}^3)\cap \mathcal{X}^2(\mathbb{R}^3)$, then $f\in\mathcal{X}^1(\mathbb{R}^3)$ and
\begin{eqnarray}\label{lem1-eqn1}
\|f\|_{\mathcal{X}^1}
&\leq &  \|f\|_{\mathcal{X}^0}^{\frac{1}{2}}\|f\|_{\mathcal{X}^2}^{\frac{1}{2}}.
\end{eqnarray}
Let $f\in\mathcal{X}_{a,\sigma}^{0}(\mathbb{R}^3)\cap \mathcal{X}_{a,\sigma}^2(\mathbb{R}^3)$,
then $f\in\mathcal{X}_{a,\sigma}^1(\mathbb{R}^3)$ and
\begin{eqnarray}\label{lem1-eqn2}
\|f\|_{\mathcal{X}_{a,\sigma}^1}
&\leq & \|f\|_{\mathcal{X}_{a,\sigma}^0}^{\frac{1}{2}}\|f\|_{\mathcal{X}_{a,\sigma}^2}^{\frac{1}{2}}.	
\end{eqnarray}
\end{lem}
{\bf Proof.} \pn
$\bullet$ {Proof of \eqref{lem1-eqn1}:} We have
\begin{eqnarray*}
\|f\|_{\mathcal{X}^1} &=&\int_{\mathbb{R}^3}|\xi||\widehat{f}(\xi)|d\xi\\
& =&\int_{\mathbb{R}^3}\left(|\widehat{f}(\xi)|\right)^{\frac{1}
{2}}\left(|\xi|^2|\widehat{f}(\xi)|\right)^{\frac{1}{2}}d\xi. 	
\end{eqnarray*}
From the Cauchy-Schwarz inequality in $L^2(\mathbb{R}^3)$, we obtain
\begin{eqnarray*}
 \|f\|_{\mathcal{X}^1} &\leq & \left(\int_{\mathbb{R}^3}|\widehat{f}(\xi)|
d\xi\right)^{\frac{1}{2}}
\left(\int_{\mathbb{R}^3}|\xi|^2|\widehat{f}(\xi)|d\xi\right)^{\frac{1}{2}}\\
&\leq & \|f\|_{\mathcal{X}^0}^{\frac{1}{2}}\|f\|_{\mathcal{X}^2}^{\frac{1}{2}}.
\end{eqnarray*}
Then $f\in\mathcal{X}^1(\mathbb{R}^3)$ and the desired result is proved.\\
$\bullet$ {Proof of \eqref{lem1-eqn2}:} We have
\begin{eqnarray*}
\|f\|_{\mathcal{X}_{a,\sigma}^1} &=&\int_{\mathbb{R}^3}|\xi| e^{a|\xi|^{\frac{1}{\sigma}}}|\widehat{f}(\xi)|d\xi \\
&=& \int_{\mathbb{R}^3} \left(e^{a|\xi|^{\frac{1}{\sigma}}} |\widehat{f}(\xi)|\right)^{\frac{1}
{2}}\left(|\xi|^2 e^{a|\xi|^{\frac{1}{\sigma}}}|\widehat{f}(\xi)|\right)^
{\frac{1}{2}}d\xi. 	
\end{eqnarray*}
The Cauchy-Schwarz inequality in $L^2(\mathbb{R}^3)$ implies
\begin{eqnarray*}
 \|f\|_{\mathcal{X}_{a,\sigma}^1} &\leq & \left(\int_{\mathbb{R}^3} e^{a|\xi|^{\frac{1}{\sigma}}}|\widehat{f}(\xi)|
d\xi\right)^{\frac{1}{2}}
\left(\int_{\mathbb{R}^3} |\xi|^2 e^{a|\xi|^{\frac{1}{\sigma}}}|\widehat{f}(\xi)|d\xi\right)^{\frac{1}{2}}\\
&\leq & \|f\|_{\mathcal{X}_{a,\sigma}^0}^{\frac{1}{2}}\|f\|_{\mathcal{X}_{a,\sigma}^2}^{\frac{1}{2}}.
\end{eqnarray*}
Then $f\in\mathcal{X}_{a,\sigma}^1(\mathbb{R}^3)$ and the desired result is proved.
\finpr
\begin{lem}\label{lem2}
 Let $f,g\in \mathcal{X}^0(\mathbb{R}^3)$. Then $fg\in \mathcal{X}^0(\mathbb{R}^3)$ and
\begin{eqnarray}
\label{lem2-eqn2}
\|fg\|_{\mathcal{X}^0}
&\leq&
\|f\|_{\mathcal{X}^0}\|g\|_{\mathcal{X}^0}.
\end{eqnarray}
Moreover, if $f\in\mathcal{X}^0(\mathbb{R}^3)$ and $k\geq 2$, then
\begin{eqnarray}
\label{lem3-eqn2}
\|f^k\|_{\mathcal{X}^0}
&\leq&
\|f\|_{\mathcal{X}^0}^k.
\end{eqnarray}
\end{lem}
{\bf Proof.} \pn
By the convolution formula and Young's inequality, we have
\begin{eqnarray*}
\|fg\|_{\mathcal{X}^0}
&=&
\int_{\mathbb{R}^3}|\widehat{fg}(\xi)|\,d\xi
\\
&\leq&
\int_{\mathbb{R}^3}\int_{\mathbb{R}^3}
|\widehat f(\xi-\eta)|\,|\widehat g(\eta)|\,d\eta\,d\xi
\\
&\leq&
\||\widehat f|\ast|\widehat g|\|_{L^1}
\\
&\leq&
\|\widehat f\|_{L^1}\|\widehat g\|_{L^1}
=
\|f\|_{\mathcal{X}^0}\|g\|_{\mathcal{X}^0}.
\end{eqnarray*}
This proves \eqref{lem2-eqn2}.\\
 Next, we prove \eqref{lem3-eqn2} by induction on $k$. For $k=2$, applying \eqref{lem2-eqn2} with $g=f$, we get
$$
\|f^2\|_{\mathcal{X}^0}
\leq
\|f\|_{\mathcal{X}^0}^{2}.
$$

Assume that, for some $k\geq 2$,
$$
\|f^k\|_{\mathcal{X}^0}
\leq
\|f\|_{\mathcal{X}^0}^{k}.
$$
Using \eqref{lem2-eqn2}, we obtain
$$
\|f^{k+1}\|_{\mathcal{X}^0}
=
\|f\,f^k\|_{\mathcal{X}^0}
\leq
\|f\|_{\mathcal{X}^0}\|f^k\|_{\mathcal{X}^0}
\leq
\|f\|_{\mathcal{X}^0}^{k+1}.
$$
Therefore,
$$
\|f^k\|_{\mathcal{X}^0}
\leq
\|f\|_{\mathcal{X}^0}^{k},
\qquad \forall k\geq 2.
$$
\finpr
\begin{lem}\label{lem22}
Let $f, g \in\mathcal{X}_{a,\sigma}^0(\mathbb{R}^3)\cap \mathcal{X}_{\frac{a}{\sigma},\sigma}^0(\mathbb{R}^3)$. Then, $fg \in  \mathcal{X}_{a,\sigma}^0(\mathbb{R}^3)$ and
\begin{eqnarray}\label{lem22-eqn1}
&& \|fg\|_{\mathcal{X}_{a,\sigma}^0} \leq  \|f\|_{\mathcal{X}_{a,\sigma}^0} \|g\|_{\mathcal{X}_{\frac{a}{\sigma},\sigma}^0}+\|f\|_{\mathcal{X}_{\frac{a}{\sigma},\sigma}^0} \|g\|_{\mathcal{X}_{a,\sigma}^0}    \\\label{lem22-eqn2}
&&\|fg\|_{\mathcal{X}_{a,\sigma}^0} \leq \|f\|_{\mathcal{X}_{a,\sigma}^0}\|g\|_{\mathcal{X}_{a,\sigma}^0}\\\nonumber
 \end{eqnarray}
\end{lem}
{\bf Proof.} \pn
$\bullet$ {Proof of \eqref{lem22-eqn1}:} We have
\begin{eqnarray*}
\|fg\|_{\mathcal{X}_{a,\sigma}^0} &=& \int_{\mathbb{R}^3} e^{a|\xi|^{\frac{1}{\sigma}}}|\widehat{fg}(\xi)|d\xi\\
&\leq &  \int_{\mathbb{R}^3} e^{a|\xi|^{\frac{1}{\sigma}}}\left(\int_{\mathbb{R}^3} |\widehat{f}(\xi-\eta)||\widehat{g}(\eta)|d\eta \right)d\xi\\
&\leq & \int_{\mathbb{R}^3} e^{a|\xi|^{\frac{1}{\sigma}}}\left(\int_{|\xi-\eta|>|\eta|} |\widehat{f}(\xi-\eta)||\widehat{g}(\eta)|d\eta +\int_{|\xi-\eta|<|\eta|} |\widehat{f}(\xi-\eta)||\widehat{g}(\eta)|d\eta \right)d\xi.
\end{eqnarray*}
Using the inequality $e^{a|\xi|^{\frac{1}{\sigma}}}\leq e^{a\max(|\xi-\eta|,|\eta|)^{\frac{1}{\sigma} }}e^{\frac{a}{\sigma}\min(|\xi-\eta|,|\eta|)^{\frac{1}{\sigma}}}$, we obtain
\begin{eqnarray*}
\|fg\|_{\mathcal{X}_{a,\sigma}^0}
&\leq &  \int_{\mathbb{R}^3} \left(\int_{|\xi-\eta|>|\eta|} e^{a|\xi-\eta|^{\frac{1}{\sigma}}}|\widehat{f}(\xi-\eta)| e^{\frac{a}{\sigma}|\eta|^{\frac{1}{\sigma}}}|\widehat{g}(\eta)|d\eta + \int_{|\xi-\eta|<|\eta|} e^{\frac{a}{\sigma}|\xi-\eta|^{\frac{1}{\sigma}}} |\widehat{f}(\xi-\eta)|e^{a|\eta|^{\frac{1}{\sigma}}}|\widehat{g}(\eta)|d\eta \right)d\xi \\
&\leq & \|( e^{a|\xi|^{\frac{1}{\sigma}}}|\widehat{f}(\xi)|)\ast  ( e^{\frac{a}{\sigma}|\xi|^{\frac{1}{\sigma}}}|\widehat{g}(\xi)|)\|_{L^1}+ \|( e^{\frac{a}{\sigma}|\xi|^{\frac{1}{\sigma}}}|\widehat{f}(\xi)|)\ast  ( e^{a|\xi|^{\frac{1}{\sigma}}}|\widehat{g}(\xi)|)\|_{L^1}\\
&\leq & \|( e^{a|\xi|^{\frac{1}{\sigma}}}|\widehat{f}(\xi)|)\|_{L^1}  \| ( e^{\frac{a}{\sigma}|\xi|^{\frac{1}{\sigma}}}|\widehat{g}(\xi)|)\|_{L^1}+ \|( e^{\frac{a}{\sigma}|\xi|^{\frac{1}{\sigma}}}|\widehat{f}(\xi)|)\|_{L^1} \| ( e^{a|\xi|^{\frac{1}{\sigma}}}|\widehat{g}(\xi)|)\|_{L^1}\\
&\leq & \|f\|_{\mathcal{X}_{a,\sigma}^0} \|g\|_{\mathcal{X}_{\frac{a}{\sigma},\sigma}^0}+\|f\|_{\mathcal{X}_{\frac{a}{\sigma},\sigma}^0} \|g\|_{\mathcal{X}_{a,\sigma}^0} .
\end{eqnarray*}

$\bullet$ {Proof of \eqref{lem22-eqn2}:} We have
\begin{eqnarray*}
\|fg\|_{\mathcal{X}_{a,\sigma}^0} &=& \int_{\mathbb{R}^3} e^{a|\xi|^{\frac{1}{\sigma}}}|\widehat{fg}(\xi)|d\xi\\
&\leq &  \int_{\mathbb{R}^3} e^{a|\xi|^{\frac{1}{\sigma}}}\left(\int_{\mathbb{R}^3} |\widehat{f}(\xi-\eta)||\widehat{g}(\eta)|d\eta \right)d\xi.	
\end{eqnarray*}
Using the inequality $e^{a|\xi|^{\frac{1}{\sigma}}}\leq e^{a|\xi-\eta|^{\frac{1}{\sigma}}}e^{a|\eta |^{\frac{1}{\sigma}}}$, we obtain
\begin{eqnarray*}
\|fg\|_{\mathcal{X}_{a,\sigma}^0}
&\leq &  \int_{\mathbb{R}^3}\left(\int_{\mathbb{R}^3} e^{a|\xi-\eta|^{\frac{1}{\sigma}}}|
\widehat{f}(\xi-\eta)| e^{a|\eta |^{\frac{1}{\sigma}}}|\widehat{g}(\eta)|d\eta \right)d\xi\\
&\leq & \|( e^{a|\xi|^{\frac{1}{\sigma}}}|\widehat{f}(\xi)|)\ast  ( e^{a|\xi|^{\frac{1}{\sigma}}}|\widehat{g}(\xi)|)\|_{L^1}\\	
&\leq & \|( e^{a|\xi|^{\frac{1}{\sigma}}}|\widehat{f}(\xi)|)\|_{L^1}\|( e^{a|\xi|^{\frac{1}{\sigma}}}|\widehat{g}(\xi)|)\|_{L^1}\\
&\leq & \|f\|_{\mathcal{X}_{a,\sigma}^0}\|g\|_{\mathcal{X}_{a,\sigma}^0}.
\end{eqnarray*}
\finpr
\begin{lem}\label{lem3} (\cite{MR2024})
Assume that $ f, g\in\mathcal{X}_{a,\sigma}^{s+1}(\mathbb{R})\cap \mathcal{X}_{\frac{a}{\sigma},\sigma}^{0}(\mathbb{R})$ with $a\geq 0, \sigma\geq 1$ and $s\geq -1$. Then $fg\in\mathcal{X}_{a,\sigma}^{s+1}(\mathbb{R})$. Moreover, the inequality below holds:
$$
\|fg\|_{\mathcal{X}_{a,\sigma}^{s+1}(\mathbb{R})}\leq C_s\Big[\|f\|_{\mathcal{X}_{\frac{a}{\sigma},\sigma}^{0}}\|g\|_{\mathcal{X}_{a,\sigma}^{s+1}}+\|f\|_{\mathcal{X}_{a,\sigma}^{s+1}}\|g\|_{\mathcal{X}_{\frac{a}{\sigma},\sigma}^{0}}\Big].
$$
 \end{lem}
\begin{lem}\label{lem33}
Let $f\in \mathcal{X}_{a,\sigma}^0(\mathbb{R}^3)\cap \mathcal{X}_{\frac{a}{\sigma},\sigma}^0(\mathbb{R}^3)$ and let $k\geq 2$. Then
\begin{eqnarray}
\label{lem33-eqn1}
\|f^k\|_{\mathcal{X}_{a,\sigma}^0}
&\leq&
2^{k-1}
\|f\|_{\mathcal{X}_{a,\sigma}^0}
\|f\|_{\mathcal{X}_{\frac{a}{\sigma},\sigma}^0}^{k-1},
\\
\label{lem33-eqn2}
\|f^k\|_{\mathcal{X}_{a,\sigma}^0}
&\leq&
2^{k-1}
\|f\|_{\mathcal{X}_{a,\sigma}^0}^{k}.
\end{eqnarray}
\end{lem}
{\bf Proof.}\pn
We first prove \eqref{lem33-eqn1} by induction on $k$.
For $k=2$, the result follows directly from \eqref{lem22-eqn1}.
Assume that, for some $k\geq 2$,
$$
\|f^k\|_{\mathcal{X}_{a,\sigma}^0}
\leq
2^{k-1}
\|f\|_{\mathcal{X}_{a,\sigma}^0}
\|f\|_{\mathcal{X}_{\frac{a}{\sigma},\sigma}^0}^{k-1}.
$$
Using \eqref{lem22-eqn1}, we get
\begin{eqnarray*}
\|f^{k+1}\|_{\mathcal{X}_{a,\sigma}^0}
&=&
\|f\,f^k\|_{\mathcal{X}_{a,\sigma}^0}
\\
&\leq&
\|f\|_{\mathcal{X}_{a,\sigma}^0}
\|f^k\|_{\mathcal{X}_{\frac{a}{\sigma},\sigma}^0}
+
\|f\|_{\mathcal{X}_{\frac{a}{\sigma},\sigma}^0}
\|f^k\|_{\mathcal{X}_{a,\sigma}^0}.
\end{eqnarray*}
By applying \eqref{lem22-eqn2} with $a$ replaced by $\frac{a}{\sigma}$, together with the
corresponding power estimate, we obtain
$$
\|f^k\|_{\mathcal{X}_{\frac{a}{\sigma},\sigma}^0}\leq\|f\|_{\mathcal{X}_{\frac{a}{\sigma},
\sigma}^0}^{k}.
$$
Moreover, by the induction hypothesis,
$$
\|f^k\|_{\mathcal{X}_{a,\sigma}^0} \leq 2^{k-1}\|f\|_{\mathcal{X}_{a,\sigma}^0}
\|f\|_{\mathcal{X}_{\frac{a}{\sigma},\sigma}^0}^{k-1}.
$$
Therefore,
$$
\begin{aligned}
\|f^{k+1}\|_{\mathcal{X}_{a,\sigma}^0}
&\leq \|f\|_{\mathcal{X}_{a,\sigma}^0}\|f^k\|_{\mathcal{X}_{\frac{a}{\sigma},\sigma}^0}
+
\|f\|_{\mathcal{X}_{\frac{a}{\sigma},\sigma}^0}\|f^k\|_{\mathcal{X}_{a,\sigma}^0}\\
&\leq\bigl(1+2^{k-1}\bigr)\|f\|_{\mathcal{X}_{a,\sigma}^0}\|f\|_{\mathcal{X}_{\frac{a}{\sigma},
\sigma}^0}^{k}\\
&\leq 2^k\|f\|_{\mathcal{X}_{a,\sigma}^0}\|f\|_{\mathcal{X}_{\frac{a}{\sigma},\sigma}^0}^{k}.
\end{aligned}
$$
Thus, \eqref{lem33-eqn1} follows by induction.\\
Now we prove \eqref{lem33-eqn2}. Since
$$
\|f\|_{\mathcal{X}_{\frac{a}{\sigma},\sigma}^0}
\leq
\|f\|_{\mathcal{X}_{a,\sigma}^0},
$$
we deduce from \eqref{lem33-eqn1} that
$$
\|f^k\|_{\mathcal{X}_{a,\sigma}^0}
\leq
2^{k-1}
\|f\|_{\mathcal{X}_{a,\sigma}^0}^{k}.
$$
Hence \eqref{lem33-eqn2} holds. This completes the proof.
\finpr
\begin{lem}\label{lem4} (\cite{BJ2026})
Let $f,g\in C_T(\mathcal{X}^0(\mathbb{R}^3))\cap L^1_T(\mathcal{X}^2(\mathbb{R}^3))$, then
$$
\int_0^t e^{(t-\tau)\Delta}\mathbb{P}\left(f\cdot \nabla g\right)d\tau \in C_T(\mathcal{X}^0(\mathbb{R}^3))\cap L^1_T(\mathcal{X}^2(\mathbb{R}^3)).
$$
Precisely, we obtain the following estimates
\begin{eqnarray}\label{lem4-eqn1}
\Big \|\int_0^t e^{(t-\tau)\Delta}\mathbb{P}\left(f\cdot \nabla g\right)d\tau \Big \|_{L^\infty_T(\mathcal{X}^0)}&\leq & T^{\frac{1}{2}}\,\|f\|_{L_T^{\infty}(\mathcal{X}^0)}\|g\|_{L_T^{\infty}
(\mathcal{X}^0)}^{\frac{1}{2}} \|g\|_{L_T^1(\mathcal{X}^2)}^{\frac{1}{2}}\\\label{lem4-eqn2}
\Big \|\int_0^t e^{(t-\tau)\Delta}\mathbb{P}\left(f\cdot \nabla g\right)d\tau \Big \|_{L^1_T(\mathcal{X}^2)}&\leq & T^{\frac{1}{2}}\,\|f\|_{L_T^{\infty}(\mathcal{X}^0)}\|g\|_{L_T^{\infty}(\mathcal{X}^0)}
^{\frac{1}{2}} \|g\|_{L_T^1(\mathcal{X}^2)}^{\frac{1}{2}}.
\end{eqnarray}
\end{lem}
\begin{lem}\label{lem44}
Let $f,g\in C_T(\mathcal{X}_{a,\sigma}^0(\mathbb{R}^3))\cap L^1_T(\mathcal{X}_{a,\sigma}^2(\mathbb{R}^3))$, then
$$
\int_0^t e^{(t-\tau)\Delta}\mathbb{P}\left(f\cdot \nabla g\right)d\tau \in C_T(\mathcal{X}_{a,\sigma}^0(\mathbb{R}^3))\cap L^1_T(\mathcal{X}_{a,\sigma}^2(\mathbb{R}^3)).
$$
Precisely, we obtain the following estimates
\begin{eqnarray} \label{lem44-eqn1}
\Big\|\int_0^t e^{(t-\tau)\Delta}\mathbb{P} (f\cdot\nabla g)d\tau \Big\|_{L_T^{\infty}(\mathcal{X}_{a,\sigma}^0)} &\leq &  T^{\frac{1}{2}}\,\|f\|_{L_T^{\infty}(\mathcal{X}_{a,\sigma}^0)}\|g\|_{L_T^{\infty}(\mathcal{X}_{a,\sigma}^0)}
^{\frac{1}{2}} \|g\|_{L_T^1(\mathcal{X}_{a,\sigma}^2)}^{\frac{1}{2}}\\\label{lem44-eqn2}
 \Big \|\int_0^t e^{(t-\tau)\Delta}\mathbb{P}\left(f\cdot \nabla g\right)d\tau \Big \|_{L^1_T(\mathcal{X}_{a,\sigma }^2)}&\leq & T^{\frac{1}{2}}\,\|f\|_{L_T^{\infty}(\mathcal{X}_{a,\sigma}^0)}\|g\|_{L_T^{\infty}(\mathcal{X}_{a,\sigma}^0)}
^{\frac{1}{2}} \|g\|_{L_T^1(\mathcal{X}_{a,\sigma}^2)}^{\frac{1}{2}}.
 \end{eqnarray}
\end{lem}
{\bf Proof.} Let $f,g\in C_T(\mathcal{X}_{a,\sigma}^0(\mathbb{R}^3))\cap L^1_T(\mathcal{X}_{a,\sigma}^2(\mathbb{R}^3))$. From  the inequality \eqref{lem1-eqn2}, we
have
\begin{eqnarray*}
\Big \|\int_0^t e^{(t-\tau)\Delta}\mathbb{P}\left(f\cdot \nabla g\right)d\tau \Big\|_{\mathcal{X}_{a,\sigma}^0}&\leq & \displaystyle\int_0^t\|f\cdot \nabla g\|_{\mathcal{X}_{a,\sigma}^0}d\tau\\
&\leq &	\displaystyle\int_0^t\|f\|_{\mathcal{X}_{a,\sigma}^0}\|g\|_{\mathcal{X}_{a,\sigma}^1}d \tau \\
&\leq &	\displaystyle\int_0^t \|f\|_{\mathcal{X}_{a,\sigma}^0} \|g\|_{\mathcal{X}_{a,\sigma}^0}^{\frac{1}{2}}\|g\|_{\mathcal{X}_{a,\sigma}^2}^{\frac{1}{2}}d\tau\\
&\leq &	\displaystyle\|f\|_{L_T^{\infty}(\mathcal{X}_{a,\sigma}^0)}\|g\|_{L_T^{\infty}(\mathcal{X}_{a,\sigma}^0)}^{\frac{1}{2}}\int_0^t\|g\|
_{\mathcal{X}_{a,\sigma}^2}^{\frac{1}{2}}d\tau.
\end{eqnarray*}
Therefore
\begin{eqnarray*}
\Big \|\int_0^t e^{(t-\tau)\Delta}\mathbb{P}\left(f\cdot \nabla g\right)d\tau \Big \|_{L_T^{\infty}(\mathcal{X}_{a,\sigma}^0)}&\leq & T^{\frac{1}{2}}\,\|f\|_{L_T^{\infty}(\mathcal{X}_{a,\sigma}^0)}\|g\|_{L_T^{\infty}
(\mathcal{X}_{a,\sigma }^0)}^{\frac{1}{2}} \|g\|_{L_T^1(\mathcal{X}_{a,\sigma}^2)}^{\frac{1}{2}},
\end{eqnarray*}
and
$$
\int_0^t e^{(t-\tau)\Delta}\mathbb{P}\left(f\cdot \nabla g\right)d\tau \in  C_T(\mathcal{X}_{a,\sigma}^0(\mathbb{R}^3)).
$$
On the other hand, we have
\begin{eqnarray*}
\Big \| \int_0^t e^{(t-\tau)\Delta}\mathbb{P}\left(f\cdot \nabla g\right)d\tau\Big \|_{L^1_T(\mathcal{X}_{a,\sigma}^2)}&\leq &  \int_0^T\int_{\rb^3}|\xi|^2\int_0^te^{-(t-\tau)|\xi|^2} e^{a|\xi|^{\frac{1}{\sigma}}}|\mathcal F(f\cdot\nabla g)(\tau,\xi)|d\tau d\xi dt\\ 	
&\leq &  \int_{\mathbb{R}^3}|\xi|^2\Big(\int_0^T\int_0^te^{-(t-\tau)|\xi|^2} e^{a|\xi|^{\frac{1}{\sigma}}}|\mathcal F(f\cdot\nabla g)
(\tau,\xi)|d\tau  dt\Big)d\xi.
\end{eqnarray*}
As $\{(\tau,t)\in[0,T]^2:\;0\leq \tau\leq t\}=\{(\tau,t)\in[0,T]^2:\; \tau\leq t\leq T\}$, we
obtain
\begin{eqnarray*}
\Big\|\int_0^t e^{(t-\tau)\Delta}\mathbb{P}\left(f\cdot \nabla g\right)d\tau\Big \|_{L^1_T(\mathcal{X}_{a,\sigma}^2)}
&\leq & \int_{\rb^3}|\xi|^2\Big(\int_0^T[\int_\tau^Te^{-(t-\tau)|\xi|
^2}dt] e^{a|\xi|^{\frac{1}{\sigma}}}|\mathcal F(f\cdot\nabla g)(\tau,\xi)|d\tau\Big)d\xi\\
&\leq &  \int_{\rb^3}|\xi|^2\Big(\int_0^T[\frac{1-e^{-(T-\tau)|\xi|^2}}{|\xi|^2}] e^{a|\xi|^{\frac{1}{\sigma}}}|\mathcal F(f\cdot\nabla g)(\tau,\xi)|d\tau\Big)d\xi\\
&\leq &  \int_{\rb^3}\Big(\int_0^T e^{a|\xi|^{\frac{1}{\sigma}}}|\mathcal F(f\cdot\nabla g)(\tau,\xi)|d\tau\Big)d\xi\\
&\leq &  \int_0^T\int_{\rb^3} e^{a|\xi|^{\frac{1}{\sigma}}}|\mathcal F(f\cdot\nabla g)(\tau,\xi)|d\xi d\tau\\
&\leq &  \int_0^T\|(f\cdot\nabla g)(\tau)\|_{\mathcal{X}_{a,\sigma}^0} d\tau.
\end{eqnarray*}
Similarly, using the inequality \eqref{lem1-eqn2} and proceeding as in the previous estimate, we obtain
\begin{eqnarray*}
\Big \| \int_0^t e^{(t-\tau)\Delta}\mathbb{P}\left(f\cdot \nabla g\right)d\tau\Big \|_{L^1_T(\mathcal{X}_{a,\sigma}^2)}
&\leq & T^{\frac{1}{2}}\,\|f\|_{L_T^{\infty}(\mathcal{X}_{a,\sigma}^0)}\|g\|_{L_T^{\infty}(\mathcal{X}_{a,\sigma}^0)}^{\frac{1}{2}}\|g\|_{L_T^1(\mathcal{X}_{a,\sigma}^2)}^{\frac{1}{2}},
\end{eqnarray*}
and
$$
\int_0^t e^{(t-\tau)\Delta}\mathbb{P}\left(f\cdot \nabla g\right)d\tau\in C_T(\mathcal{X}_{a,\sigma}^0(\mathbb{R}^3))\cap L^1_T(\mathcal{X}_{a,\sigma}^2(\mathbb{R}^3))
$$
which completes the proof.
\finpr
\begin{rem}
In what follows, $e_0$ denotes an arbitrary vector of the canonical basis $\{e_1,e_2,e_3\}$ of
$\mathbb{R}^3$. In particular, $|e_0|=1$, so multiplication by $e_0$ does not affect the
Fourier-space norms appearing below.
\end{rem}
\begin{lem}\label{lem5}
Let $f_1, f_2, f_3\in C_T(\mathcal{X}^0(\mathbb{R}^3))\cap L^1_T(\mathcal{X}^2(\mathbb{R}^3))$. Then, for $p,q,r\in\mathbb{N}_0$, we have
$$
\int_0^t e^{(t-\tau)\Delta}\mathbb{P}( f_1^p\,f_2^q\,f_3^re_0)d\tau\in  C_T(\mathcal{X}^0(\mathbb{R}^3))\cap L^1_T(\mathcal{X}^2(\mathbb{R}^3)),
$$
Precisely, we obtain the following estimates
 \begin{eqnarray}\label{lem5-eqn1}
&&\Big\|\int_0^t e^{(t-\tau)\Delta}\mathbb{P} ( f_1^p\,f_2^q\,f_3^re_0)d\tau \Big\|_{L_T^{\infty}(\mathcal{X}^0)} \leq T \|f_1\|_{L_T^{\infty}(\mathcal{X}^0)}^p\|f_2\|_{L_T^{\infty}(\mathcal{X}^0)}^q\|f_3\|_{L_T^{\infty}(\mathcal{X}^0)}^r\\\label{lem5-eqn2}
&& \Big\|\int_0^t e^{(t-\tau)\Delta}\mathbb{P}(f_1^p\,f_2^q\,f_3^re_0)d\tau \Big\|_{L_T^1(\mathcal{X}^2)} \leq T \|f_1\|_{L_T^{\infty}(\mathcal{X}^0)}^p\|f_2\|_{L_T^{\infty}(\mathcal{X}^0)}^q\|f_3\|_{L_T^{\infty}(\mathcal{X}^0)}^r
 \end{eqnarray}
Particularly: For $f^k\in C_T(\mathcal{X}^0(\mathbb{R}^3))\cap L^1_T(\mathcal{X}^2(\mathbb{R}^3))$ with $k\geq 2$, we have
$$
\int_0^t e^{(t-\tau)\Delta}\mathbb{P}( f^ke_0)d\tau\in   C_T(\mathcal{X}^0(\mathbb{R}^3))\cap L^1_T(\mathcal{X}^2(\mathbb{R}^3)),
$$
and the following estimates hold
 \begin{eqnarray}\label{lem5-eqn3}
&&\Big\|\int_0^t e^{(t-\tau)\Delta}\mathbb{P} (f^ke_0)d\tau \Big\|_{L_T^{\infty}(\mathcal{X}^0)} \leq T \|f\|_{L_T^{\infty}(\mathcal{X}^0)}^k\\\label{lem5-eqn4}
&& \Big\|\int_0^t e^{(t-\tau)\Delta}\mathbb{P} (f^ke_0)d\tau \Big\|_{L_T^1(\mathcal{X}^2)} \leq T \|f\|_{L_T^{\infty}(\mathcal{X}^0)}^k
 \end{eqnarray}
\end{lem}
{\bf Proof.}\pn
 We start by proving $ \int_0^t e^{(t-\tau)\Delta}\mathbb{P}( f_1^p\,f_2^q\,f_3^re_0)d\tau\in  C_T(\mathcal{X}^0(\mathbb{R}^3)) :$ Using twice the inequality \eqref{lem2-eqn2} and the inequality \eqref{lem3-eqn2}, we obtain
\begin{eqnarray*}
\Big\|\int_0^t e^{(t-\tau)\Delta}\mathbb{P} (f_1^p\,f_2^q\,f_3^re_0)d\tau \Big\|_{\mathcal{X}^0}
&\leq &   \int_0^t \Big\| e^{(t-\tau)\Delta}\mathbb{P} (f_1^p\,f_2^q\,f_3^re_0) \Big\|_{\mathcal{X}^0}d\tau \\
&\leq & \int_0^t \Big\| e^{(t-\tau)\Delta} (f_1^p\,f_2^q\,f_3^r) \Big\|_{\mathcal{X}^0}d\tau \\
&\leq & \int_0^t\left( \int_{\mathbb{R}^3}  e^{- (t-\tau)|\xi|^2}  |\mathcal{F}(f_1^p\,f_2^q\,f_3^r)(\xi)|d\xi \right)d\tau \\
&\leq & \int_0^t\|f_1^p\,f_2^q\,f_3^r\|_{\mathcal{X}^0}d\tau \\
&\leq & T\|f_1\|_{L_T^{\infty}(\mathcal{X}^0)}^p\|f_2\|_{L_T^{\infty}(\mathcal{X}^0)}^q\|f_3 \|_{L_T^{\infty}(\mathcal{X}^0)}^r.
\end{eqnarray*}
Secondly, we prove that $ \int_0^t e^{(t-\tau)\Delta}\pb( f_1^p\,f_2^q\,f_3^re_0)d\tau\in
   L^1_T(\mathcal{X}^2(\mathbb{R}^3))$: We have
\begin{eqnarray*}
\Big\|\int_0^t e^{(t-\tau)\Delta}\mathbb{P}( f_1^p\,f_2^q\,f_3^re_0)d\tau\Big \|_{L^1_T(\mathcal{X}^2)}
&\leq &  \int_0^T\int_{\mathbb{R}^3}|\xi|^2\int_0^te^{-(t-\tau)|\xi|^2}|
\mathcal F(f_1^p\,f_2^q\,f_3^r)(\tau,\xi)|d\tau d\xi dt\\ 	
&\leq &  \int_{\mathbb{R}^3}|\xi|^2\Big(\int_0^T\int_0^te^{-(t-\tau)|\xi|^2}|\mathcal
F(f_1^p\,f_2^q\,f_3^r)(\tau,\xi)|d\tau  dt\Big)d\xi.
\end{eqnarray*}
As $\{(\tau,t)\in[0,T]^2:\;0\leq \tau\leq t\}=\{(\tau,t)\in[0,T]^2:\; \tau\leq t\leq T\}$, we
obtain
\begin{eqnarray*}
\Big \|\int_0^t e^{(t-\tau)\Delta}\mathbb{P}(f_1^p\,f_2^q\,f_3^re_0)d\tau\Big\|_{L^1_T(\mathcal{X}^2)}
&\leq &  \int_{\mathbb{R}^3}|\xi|^2\Big(\int_0^T[\int_\tau^Te^{-(t-\tau)|
\xi|^2}dt]|\mathcal F(f_1^p\,f_2^q\,f_3^r)(\tau,\xi)|d\tau\Big)d\xi\\
&\leq &  \int_{\mathbb{R}^3}|\xi|^2\Big(\int_0^T[\frac{1-e^{-(T-\tau)|\xi|^2}}{|\xi|^2}]|
\mathcal F(f_1^p\,f_2^q\,f_3^r)(\tau,\xi)|d\tau\Big)d\xi\\
&\leq &\int_{\mathbb{R}^3}\Big(\int_0^T|\mathcal F(f_1^p\,f_2^q\,f_3^r)(\tau,\xi)|d\tau\Big)d\xi\\
&\leq &  \int_0^T\int_{\mathbb{R}^3}|\mathcal F(f_1^p\,f_2^q\,f_3^r)(\tau,\xi)|d\xi d\tau\\
&\leq &  \int_0^T\|(f_1^p\,f_2^q\,f_3^r)(\tau)\|_{\mathcal{X}^0} d\tau.
\end{eqnarray*}
Similarly, using twice the inequality \eqref{lem2-eqn2} and the inequality \eqref{lem3-eqn2}, we obtain
\begin{eqnarray*}
\Big\| \int_0^t e^{(t-\tau)\Delta}\mathbb{P}( f_1^p\,f_2^q\,f_3^re_0)d\tau\Big\|_{L^1_T(\mathcal{X}^2)}
&\leq &T\,\|f_1\|_{L_T^\infty(\mathcal{X}^0)}^p\|f_2\|_{L_T^\infty(\mathcal{X}^0)}^q\|f_3\|_{L_T^\infty(\mathcal{X}^0)}^r.
\end{eqnarray*}
Finally, the particular case follows by taking $f_1=f_2=f_3=f$ and choosing $p,q,r\in\mathbb{N}_0$ such that $p+q+r=k$. Thus,
$$
f_1^p f_2^q f_3^r=f^{p+q+r}=f^k.
$$
Applying estimates \eqref{lem5-eqn1} and \eqref{lem5-eqn2}, we obtain
$$
\Big\|\int_0^t e^{(t-\tau)\Delta}\mathbb{P}(f^ke_0)d\tau
\Big\|_{L_T^{\infty}(\mathcal{X}^0)}
\leq
T\|f\|_{L_T^\infty(\mathcal{X}^0)}^k,
$$
and
$$
\Big\|\int_0^t e^{(t-\tau)\Delta}\mathbb{P}(f^ke_0)d\tau
\Big\|_{L_T^1(\mathcal{X}^2)}
\leq
T\|f\|_{L_T^\infty(\mathcal{X}^0)}^k.
$$
Therefore,
$$
\int_0^t e^{(t-\tau)\Delta}\mathbb{P}(f^ke_0)d\tau
\in
C_T(\mathcal{X}^0(\mathbb{R}^3))
\cap
L_T^1(\mathcal{X}^2(\mathbb{R}^3)).
$$
This completes the proof.
\qed
\begin{lem}\label{lem55}
Let $f_1, f_2, f_3\in C_T(\mathcal{X}_{a,\sigma}^0(\mathbb{R}^3))\cap L^1_T(\mathcal{X}_{a,\sigma}^2(\mathbb{R}^3))$. Then, for $p,q,r\in\mathbb{N}_0$, we have
$$
\int_0^t e^{(t-\tau)\Delta}\mathbb{P}( f_1^p\,f_2^q\,f_3^re_0)d\tau\in  C_T(\mathcal{X}_{a,\sigma}^0(\mathbb{R}^3))\cap L^1_T(\mathcal{X}_{a,\sigma }^2(\mathbb{R}^3)),
$$
Precisely, we obtain the following estimates
 \begin{eqnarray}\label{lem55-eqn1}
&&\Big\|\int_0^t e^{(t-\tau)\Delta}\mathbb{P} (f_1^p\,f_2^q\,f_3^r e_0)d\tau \Big\|_{L_T^{\infty}(\mathcal{X}_{a,\sigma}^0)} \leq T \|f_1\|_{L_T^{\infty}(\mathcal{X}_{a,\sigma}^0)}^p\|f_2\|_{L_T^{\infty}(\mathcal{X}_{a,\sigma }^0)}^q\|f_3\|_{L_T^{\infty}(\mathcal{X}_{a,\sigma }^0)}^r\\\label{lem55-eqn2}
&& \Big\|\int_0^t e^{(t-\tau)\Delta}\mathbb{P} (f_1^p\,f_2^q\,f_3^r e_0)d\tau \Big\|_{L_T^1(\mathcal{X}_{a,\sigma}^2)} \leq T \|f_1\|_{L_T^{\infty}(\mathcal{X}_{a,\sigma}^0)}^p\|f_2\|_{L_T^{\infty}(\mathcal{X}_{a,\sigma}^0)}^q\|f_3\|_{L_T^{\infty}(\mathcal{X}_{a,\sigma}^0)}^r.
 \end{eqnarray}
En particular: For $f^k\in C_T(\mathcal{X}_{a,\sigma }^0(\mathbb{R}^3))\cap L^1_T(\mathcal{X}_{a,\sigma}^2(\mathbb{R}^3))$ with $k\geq 2$, we have
$$
\int_0^t e^{(t-\tau)\Delta}\mathbb{P}( f^ke_0)d\tau\in   C_T(\mathcal{X}_{a,\sigma }^0(\mathbb{R}^3))\cap L^1_T(\mathcal{X}_{a,\sigma}^2(\mathbb{R}^3)),
$$
and the following estimates hold
 \begin{eqnarray}\label{lem55-eqn3}
&&\Big\|\int_0^t e^{(t-\tau)\Delta}\mathbb{P} (f^ke_0)d\tau \Big\|_{L_T^{\infty}(\mathcal{X}_{a,\sigma}^0)} \leq T \|f\|_{L_T^{\infty}(\mathcal{X}_{a,\sigma}^0)}^k\\\label{lem55-eqn4}
&& \Big\|\int_0^t e^{(t-\tau)\Delta}\mathbb{P} (f^ke_0)d\tau \Big\|_{ L_T^1(\mathcal{X}_{a,\sigma }^2)} \leq T \|f\|_{L_T^{\infty}(\mathcal{X}_{a,\sigma}^0)}^k.
 \end{eqnarray}
\end{lem}
{\bf Proof.}\pn
 We start by proving $ \int_0^t e^{(t-\tau)\Delta}\pb(  f_1^p\,f_2^q\,f_3^re_0)d\tau\in  C_T(\mathcal{X}_{a,\sigma }^0(\mathbb{R}^3)) :$
Using twice the inequality \eqref{lem22-eqn2} and the inequality \eqref{lem33-eqn2}, we obtain
\begin{eqnarray*}
\Big\|\int_0^t e^{(t-\tau)\Delta}\mathbb{P} (f_1^p\,f_2^q\,f_3^re_0)d\tau \Big\|_{\mathcal{X}_{a,\sigma}^0}
&\leq &   \int_0^t \Big\| e^{(t-\tau)\Delta}\mathbb{P} (f_1^p\,f_2^q\,f_3^re_0) \Big\|_{\mathcal{X}_{a,\sigma}^0}d\tau \\
&\leq & \int_0^t \Big\| e^{(t-\tau)\Delta} (f_1^p\,f_2^q\,f_3^r) \Big\|_{\mathcal{X}_{a,\sigma}^0}d\tau \\
&\leq & \int_0^t\left( \int_{\mathbb{R}^3}  e^{-(t-\tau)|\xi|^2} e^{a|\xi|^{\frac{1}{\sigma}}} |\mathcal{F}(f_1^p\,f_2^q\,f_3^r)(\xi)|d\xi \right)d\tau \\
&\leq & \int_0^t\|f_1^p\,f_2^q\,f_3^r\|_{\mathcal{X}_{a,\sigma }^0}d\tau \\
&\leq & T\|f_1\|_{L_T^{\infty}(\mathcal{X}_{a,\sigma }^0)}^p\|f_2\|_{L_T^{\infty}(\mathcal{X}_{a,\sigma}^0)}^q\|f_3 \|_{L_T^{\infty}(\mathcal{X}_{a,\sigma}^0)}^r.
\end{eqnarray*}
Secondly, we prove that $ \int_0^t e^{(t-\tau)\Delta}\pb( f_1^p\,f_2^q\,f_3^re_0)d\tau\in
  L^1_T(\mathcal{X}_{a,\sigma }^2(\mathbb{R}^3))$: We have
\begin{eqnarray*}
\Big\|\int_0^t e^{(t-\tau)\Delta}\mathbb{P}( f_1^p\,f_2^q\,f_3^r e_0)d\tau\Big \|_{L^1_T(\mathcal{X}_{a,\sigma }^2)}
&\leq &  \int_0^T\int_{\mathbb{R}^3}|\xi|^2\int_0^te^{-(t-\tau)|\xi|^2} e^{a|\xi|^{\frac{1}{\sigma}}}|\mathcal F(f_1^p\,f_2^q\,f_3^r)(\tau,\xi)|d\tau d\xi dt\\ 	
&\leq &  \int_{\mathbb{R}^3}|\xi|^2\Big(\int_0^T\int_0^te^{-(t-\tau)|\xi|^2} e^{a|\xi|^{\frac{1}{\sigma}}}|\mathcal
F(f_1^p\,f_2^q\,f_3^r)(\tau,\xi)|d\tau  dt\Big)d\xi.
\end{eqnarray*}
As $\{(\tau,t)\in[0,T]^2:\;0\leq \tau\leq t\}=\{(\tau,t)\in[0,T]^2:\; \tau\leq t\leq T\}$, we
obtain
\begin{eqnarray*}
\Big \|\int_0^t e^{(t-\tau)\Delta}\mathbb{P}(f_1^p\,f_2^q\,f_3^re_0)d\tau\Big\|_{L^1_T(\mathcal{X}_{a,\sigma}^2)}
&\leq & \int_{\mathbb{R}^3}|\xi|^2\Big(\int_0^T[\int_\tau^Te^{-(t-\tau)|
\xi|^2}dt]e^{a|\xi|^{\frac{1}{\sigma}}}|\mathcal F(f_1^p\,f_2^q\,f_3^r)(\tau,\xi)|d\tau\Big)d\xi\\
&\leq &  \int_{\mathbb{R}^3}|\xi|^2\Big(\int_0^T[\frac{1-e^{-(T-\tau)|\xi|^2}}{|\xi|^2}]
e^{a|\xi|^{\frac{1}{\sigma}}}|\mathcal F(f_1^p\,f_2^q\,f_3^r)(\tau,\xi)|d\tau\Big)d\xi\\
&\leq &\int_{\mathbb{R}^3}\Big(\int_0^T e^{a|\xi|^{\frac{1}{\sigma}}}|\mathcal F(f_1^p\,f_2^q\,f_3^r)(\tau,\xi)|d\tau\Big)d\xi\\
&\leq &  \int_0^T\int_{\mathbb{R}^3} e^{a|\xi|^{\frac{1}{\sigma}}}|\mathcal F(f_1^p\,f_2^q\,f_3^r)(\tau,\xi)|d\xi d\tau\\
&\leq &  \int_0^T\|f_1^p\,f_2^q\,f_3^r(\tau)\|_{\mathcal{X}_{a,\sigma }^0} d\tau.
\end{eqnarray*}
Similarly, using twice the inequality \eqref{lem22-eqn2} and the inequality \eqref{lem33-eqn2}, we obtain
\begin{eqnarray*}
\Big\| \int_0^t e^{(t-\tau)\Delta}\mathbb{P}( f_1^p\,f_2^q\,f_3^re_0)d\tau\Big\|_{L^1_T(\mathcal{X}_{a,\sigma}^2)}
&\leq &T\,\|f_1\|_{L_T^\infty(\mathcal{X}_{a,\sigma}^0)}^p\|f_2\|_{L_T^\infty
(\mathcal{X}_{a,\sigma}^0)}^q\|f_3\|_{L_T^\infty(\mathcal{X}_{a,\sigma }^0)}^r.
\end{eqnarray*}
Finally, the particular case follows by taking $f_1=f_2=f_3=f$ and choosing $p,q,r\in\mathbb{N}_0$ such that $p+q+r=k$. Thus,
$$
f_1^p f_2^q f_3^r=f^{p+q+r}=f^k.
$$
Applying the previous estimates, we get
$$
\Big\|\int_0^t e^{(t-\tau)\Delta}\mathbb{P}(f^ke_0)d\tau
\Big\|_{L_T^{\infty}(\mathcal{X}_{a,\sigma}^0)}
\leq
T\|f\|_{L_T^\infty(\mathcal{X}_{a,\sigma}^0)}^k,
$$
and
$$
\Big\|\int_0^t e^{(t-\tau)\Delta}\mathbb{P}(f^ke_0)d\tau
\Big\|_{L_T^1(\mathcal{X}_{a,\sigma}^2)}
\leq
T\|f\|_{L_T^\infty(\mathcal{X}_{a,\sigma}^0)}^k.
$$
Therefore,
$$
\int_0^t e^{(t-\tau)\Delta}\mathbb{P}(f^ke_0)d\tau
\in
C_T(\mathcal{X}_{a,\sigma}^0(\mathbb{R}^3))
\cap
L_T^1(\mathcal{X}_{a,\sigma}^2(\mathbb{R}^3)).
$$
This completes the proof.
 \finpr
\section{\bf Proof of Theorem \ref{theo1}}
Let $R$ be a positive real number. For $T>0$, we define the Banach space
$$
E_T=\mathcal{C}_{T}\bigl(\X^{0}(\mathbb{R}^{3})\bigr)
\cap L^{1}_{T}\bigl(\X^{2}(\mathbb{R}^{3})\bigr)
$$
equipped with norm
$$
\|u\|_{E_T}=\|u\|_{L_T^{\infty}(\X^0)}+\|u\|_{L_T^1(\X^2)}.
$$
We denote the closed subset $F_{R, T}$ of $E_T$ defined by
$$
F_{R, T}=\Big\{u\in E_T; \|u\|_{L_T^{\infty}(\mathcal{X}^0)}\leq 2\|u^0\|_{\X^0},\;\|u\|_{L_T^1(\X^2)}\leq R\Big\}.
$$
Now, consider the operator
\begin{eqnarray*}
\Psi\,&:&\,F_{R, T}\longrightarrow	 E_T\\
               && u\longmapsto \Psi (u)= e^{t\Delta}u^0-\mathcal{B}(u,u)(t)-\alpha \mathcal{D}_m(u)(t).
\end{eqnarray*}
To apply the fixed-point theorem, it suffices to prove that the following estimates hold
for $T$ and $R$ sufficiently small:
\begin{eqnarray}\label{eq31}
&&\Psi(F_{R, T})\subset F_{R, T}\\\label{eq32}
&& \|\Psi(u)-\Psi(v)\|_{E_T}\leq \frac{1}{2}\|u-v\|_{E_T},\,\,\forall u,v\in F_{R,T}.
\end{eqnarray}
First, we prove \eqref{eq31}: \pn
$\bullet$ To estimate $\Psi(u)$ in $\mathcal{X}^0(\rb^3)$, we have
\begin{eqnarray*}
\| e^{t\Delta}u^0\|_{L_T^{\infty}(\mathcal{X}^0)} &\leq & \|u^0\|_{\mathcal{X}^0}.\hskip 3cm
\end{eqnarray*}
 Using the inequalities \eqref{lem4-eqn1}-\eqref{lem5-eqn3} and the fact that $u\in F_{R,T}$, we obtain
\begin{eqnarray*}
\|\mathcal{B}(u,u)\|_{L_T^{\infty}(\mathcal{X}^0)}&\leq &T^{\frac{1}{2}}\|u\|_{L_T^{\infty}(\mathcal{X}^0)}^{\frac{3}{2}}\|u\|_{L_T^{1}(\mathcal{X}^2)}^{\frac{1}{2}} \\
&\leq &  2\sqrt{2}T^{\frac{1}{2}}R^{\frac{1}{2}}\|u^0\|_{\mathcal{X}^0}^{\frac{3}{2}}
\end{eqnarray*}
and
\begin{eqnarray*}
\| \mathcal{D}_m(u)\|_{L_T^{\infty}(\mathcal{X}^0)}
&\leq & T\|u\|_{L_T^{\infty}(\mathcal{X}^0)}^{2m+1}\\
&\leq & 2^{2m+1}  T\|u^0\|_{\mathcal{X}^0}^{2m+1}.
\end{eqnarray*}
Then, there is a time $T_1>0$ such that
$$
  2\sqrt{2}T^{\frac{1}{2}}R^{\frac{1}{2}}\|u^0\|_{\mathcal{X}^0}^{\frac{3}{2}}+ 2^{2m+1}\alpha T\|u^0\|_{\mathcal{X}^0}^{2m+1} \leq \|
u^0\|_{\mathcal{X}^0},\,\,\forall \,0<T\leq T_1.
$$
By using this choice, we obtain
\begin{eqnarray}\label{eq33}
\|\Psi(u)\|_{L_T^{\infty}(\mathcal{X}^0)}\leq 2\|u^0\|_{\mathcal{X}^0},\,\forall u\in 	F_{R, T}.
\end{eqnarray}
\noindent $\bullet$ To estimate $\Psi(u)$ in $L^1_T(\X^2(\rb^3))$, we have
\begin{eqnarray*}
\|e^{t\Delta}u^0\|_{L_T^{1}(\X^2)} &= & \int_0^T\int_{\xi}|\xi|^2e^{-t|\xi|^2}|\wh{u^0(\xi)}|
d\xi dt\\
&=&\int_{\xi}\left(1-e^{-T|\xi|^2}\right)|\wh{u^0(\xi)}|d\xi.
\end{eqnarray*}
By dominated convergence theorem, there is a time $T_2>0$ such that
$$
\|e^{t\Delta}u^0\|_{L_T^{1}(\X^2)}<\frac{R}{3},\;\forall\,0<T\leq T_2.
$$
By inequality \eqref{lem4-eqn2}, we get
\begin{eqnarray*}
\|\mathcal{B}(u,u)\|_{L_T^{1}(\X^2)} &\leq &T^{\frac{1}{2}}\|u\|_{L_T^{\infty}(\mathcal{X}^0)}^{\frac{3}{2}}\|
u\|_{L_T^{1}(\X^2)}^{\frac{1}{2}}\\
&\leq &2\sqrt{2}T^{\frac{1}{2}} R^{\frac{1}{2}}\|u^0\|_{\mathcal{X}^0}^{\frac{3}{2}}.
\end{eqnarray*}
By inequality \eqref{lem5-eqn4}, we have
\begin{eqnarray*}
\|\mathcal{D}_m(u)\|_{L_T^{1}(\X^2)}&\leq& T\|u\|_{L_T^{\infty}(\mathcal{X}^0)}^{2m+1}\\
&\leq& 2^{2m+1} T\|u^0\|_{\mathcal{X}^0}^{2m+1}.
\end{eqnarray*}
Therefore, there exists a time $T_3>0$ such that
$$
2\sqrt{2}T^{\frac12}R^{\frac12}
\|u^0\|_{\mathcal{X}^0}^{\frac32}
+
2^{2m+1}\alpha T
\|u^0\|_{\mathcal{X}^0}^{2m+1}
\leq \frac{2R}{3},
\qquad \forall\,0<T\leq T_3.
$$
Hence,
\begin{equation}\label{eq34}
\|\Psi(u)\|_{L_T^1(\mathcal{X}^2)}
\leq R,
\qquad \forall u\in F_{R,T}.
\end{equation}
Then, choosing
$$
0<T\leq \min\{T_1,T_2,T_3\},
$$
we deduce from \eqref{eq33} and \eqref{eq34} that
$$
\Psi(F_{R,T})\subset F_{R,T}.
$$
Second, we prove \eqref{eq32}:
By using the fact that
\begin{eqnarray}\label{blow-10}
u\cdot \nabla u-v\cdot \nabla v=u\cdot\nabla (u-v)+(u-v)\cdot\nabla v	
\end{eqnarray}
From inequality \eqref{lem4-eqn1}, for $u,v\in F_{R, T}$, we have
\begin{eqnarray}\nonumber
\|\mathcal{B}(u,u)-\mathcal{B}(v,v)\|_{L_T^{\infty}(\mathcal{X}^0)}
&\leq & T^{\frac{1}{2}}\,\|u\|_{L_T^{\infty}(\mathcal{X}^0)}
\|u-v\|_{L_T^{\infty}(\mathcal{X}^0)}^{\frac{1}{2}}\|u-v\|_{L_T^{1}(\mathcal{X}^2)}^{\frac{1}{2}}\\\nonumber
&& + T^{\frac{1}{2}}\,\|u-v\|_{L_T^\infty(\X^0)}\|v\|_{L_T^{\infty}(\X^0)}^{\frac{1}{2}}\|v\|_{L_T^{1}(\X^2)}^{\frac{1}{2}}\\\nonumber
&\leq& 2 T^{\frac{1}{2}}\,\|u^0\|_{\X^0}\|u-v\|_{E_T}+\sqrt{2} T^{\frac{1}{2}}R^{\frac{1}{2}}\|u^0\|_{\X^0}^{\frac{1}{2}} \|u-v\|_{E_T}\\\label{eq35}
&\leq &\sqrt{2}T^{\frac{1}{2}}\|u^0\|_{\X^0}^{\frac{1}{2}}\left(\sqrt{2}\|u^0\|_{\X^0}^{\frac{1}{2}}+R^{\frac{1}{2}}\right)\|u-v\|_{E_T}.
\end{eqnarray}
By using inequality \eqref{lem4-eqn2} , we obtain
\begin{eqnarray}\nonumber
\|\mathcal{B}(u,u)-\mathcal{B}(v,v)\|_{L_T^1(\X^2)}	
&\leq & T^{\frac{1}{2}}\,\|u\|_{L_T^{\infty}(\X^0)}
\|u-v\|_{L_T^{\infty}(\X^0)}^{\frac{1}{2}}\|u-v\|_{L_T^{1}(\X^2)}^{\frac{1}{2}}\\\nonumber
&& + T^{\frac{1}{2}}\,\|u-v\|_{L_T^\infty(\X^0)}\|v\|_{L_T^{\infty}(\X^0)}^{\frac{1}{2}}\|v\|_{L_T^{1}
(\X^2)}^{\frac{1}{2}}\\\nonumber
&\leq & 2 T^{\frac{1}{2}}\,\|u^0\|_{\X^0}\|u-v\|_{E_T}+\sqrt{2}T^{\frac{1}{2}} R^{\frac{1}{2}}\|u^0\|_{\X^0}^{\frac{1}{2}} \|u-v\|_{E_T}\\\label{eq36}
&\leq &\sqrt{2} T^{\frac{1}{2}}\|u^0\|_{\X^0}^{\frac{1}{2}}\left(\sqrt{2}\|u^0\|_{\X^0}^{\frac{1}{2}}+R^{\frac{1}{2}}\right)\|u-v\|_{E_T}.
\end{eqnarray}
Now, using the fact that
\begin{eqnarray}\nonumber
u_k^{2m+1}-v_k^{2m+1} &=& (u_k-v_k)\sum_{j=0}^{2m}u_k^{2m-j}v_k^j\\
&=& (u_k-v_k) \left(u_k^{2m}+u_k^{2m-1}v_k+...+u_kv_k^{2m-1}+v_k^{2m}\right),\label{blow-11}
\end{eqnarray}
we get
\begin{eqnarray}\nonumber
\mathcal{D}_m(u)-\mathcal{D}_m(v) &=& -\sum_{k=1}^3\int_0^te^{(t-\tau)\Delta }\pb  (u_k^{2m+1}-v_k^{2m+1})e_kd\tau\\
&=&  -\sum_{k=1}^3\int_0^te^{(t-\tau)\Delta }\pb  \left( (u_k-v_k) \left(u_k^{2m}+u_k^{2m-1}v_k+...+u_kv_k^{2m-1}+v_k^{2m}\right)\right)e_kd\tau.\label{blow-12}
\end{eqnarray}
By using inequality \eqref{lem5-eqn1}, we get
\begin{eqnarray}\nonumber
\|\mathcal{D}_m(u)-\mathcal{D}_m(v)\|_{L_T^{\infty}(\X^0)}
&\leq & T\Big[\|u\|_{L_T^{\infty}(\X^0)}^{2m}+ \|u\|_{L_T^{\infty}(\X^0)}^{2m-1}\|v\|_{L_T^{\infty}(\X^0)}+..\\\nonumber
&&..+ \|u\|_{L_T^{\infty}(\X^0)}\|v\|_{L_T^{\infty}(\X^0)}^{2m-1}+\|v\|_{L_T^{\infty}(\X^0)}^{2m}\Big]\|u-v\|_{L_T^{\infty}(\X^0)}\\\nonumber
&\leq & (2m+1)T\Big[\|u\|_{L_T^{\infty}(\X^0)}^{2m}+ \|v\|_{L_T^{\infty}(\X^0)}^{2m}\Big]\|u-v\|_{L_T^{\infty}(\X^0)}\\ \label{eq37}
&\leq & (2m+1)2^{2m+1}T\|u^0\|_{\X^0}^{2m}\|u-v\|_{E_T}.
\end{eqnarray}
By using inequalities \eqref{lem5-eqn2}, we get
\begin{eqnarray}\nonumber
\|\mathcal{D}_m(u)-\mathcal{D}_m(v)\|_{L_T^1(\X^2)}
&\leq &  T\Big[\|u\|_{L_T^{\infty}(\X^0)}^{2m}+ \|u\|_{L_T^{\infty}(\X^0)}^{2m-1}\|v\|_{L_T^{\infty}(\X^0)}+...\\\nonumber
&&...+ \|u\|_{L_T^{\infty}(\X^0)}\|v\|_{L_T^{\infty}(\X^0)}^{2m-1}+\|v\|_{L_T^{\infty}(\X^0)}^{2m}\Big]\|u-v\|_{L_T^{\infty}(\X^0)}\\\nonumber
&\leq &   (2m+1)T\Big[\|u\|_{L_T^{\infty}(\X^0)}^{2m}+\|v\|_{L_T^{\infty}(\X^0)}^{2m}\Big]\|u-v\|_{L_T^{\infty}(\X^0)}\\\label{eq38}
&\leq & (2m+1)2^{2m+1}T\|u^0\|_{\X^0}^{2m}\|u-v\|_{E_T}
\end{eqnarray}
Finally, by decreasing $T>0$ if necessary, while keeping
$$
T\leq \min\{T_1,T_2,T_3\},
$$
we may assume that
\begin{eqnarray*}
&& \sqrt{2} T^{\frac{1}{2}}\|u^0\|_{\X^0}^{\frac{1}{2}}\left(\sqrt{2}\|u^0\|_{\X^0}^{\frac{1}{2}}+R^{\frac{1}{2}}\right)\leq \frac{1}{8}, \\
&& (2m+1)2^{2m+1} \alpha T\|u^0\|_{\X^0}^{2m}\leq \frac{1}{8}
\end{eqnarray*}
Combining \eqref{eq35}--\eqref{eq38} with the above choices of $T$, we obtain
$$
\|\Psi(u)-\Psi(v)\|_{E_T}\leq\frac12\|u-v\|_{E_T}.
$$
Hence, $\Psi$ is a contraction on $F_{R,T}$.\\
Therefore, the Banach fixed-point theorem yields a unique local solution to
$(NSE_M)$ in
$$
\mathcal{C}_{T}(\mathcal{X}^{0}(\mathbb{R}^{3}))
\cap
L^{1}_{T}(\mathcal{X}^{2}(\mathbb{R}^{3})).
$$
This completes the proof of Theorem~\ref{theo1}.

\section{\bf Proof of Theorem \ref{theo2}}

Let $u\in C([0,T^*),\mathcal{X}^0(\mathbb{R}^3)) \cap L^1_{\mathrm{loc}}([0,T^*),\mathcal{X}
^2(\mathbb{R}^3))$ be the maximal solution of $(NSE_M)$ given by Theorem~\ref{theo1}.
From the first equation of the system $(NSE_M)$, we have
\begin{eqnarray*}
&& \|u(t)\|_{\mathcal{X}^0}+\int_0^t\|u(z)\|_{\mathcal{X}^2}dz\\
&& \quad \leq  \|u^0\|_{\mathcal{X}^0}+\int_0^t\|u(z)\|_{\mathcal{X}^0}\|u(z)\|_{\mathcal{X}^1}dz+\alpha\int_0^t\|u(z)\|_{\mathcal{X}^0}^{2m+1}dz.	
\end{eqnarray*}
By using the inequality \eqref{lem1-eqn1}, we get
\begin{eqnarray*}
&& \|u(t)\|_{\mathcal{X}^0}+\int_0^t\|u(z)\|_{\mathcal{X}^2}dz\\
&& \quad \leq  \|u^0\|_{\mathcal{X}^0}+\int_0^t\|u(z)\|_{\mathcal{X}^0}^{\frac{3}{2}}\|u(z)\|_{\mathcal{X}^2}^{\frac{1}{2}}dz+\alpha\int_0^t\|u(z)\|_{\mathcal{X}^0}^{2m+1}dz\\	
\end{eqnarray*}
The elementary inequality
$$
xy\leq \frac{x^2}{2}+\frac{y^2}{2},\;\forall x,y\geq0,
$$
implies that
\begin{eqnarray*}
\|u(t)\|_{\mathcal{X}^0}+\frac{1}{2}\int_0^t\|u(z)\|_{\mathcal{X}^2}dz
&\leq &\|u^0\|_{\mathcal{X}^0}+\frac{1}{2}\int_0^t\|u(z)\|_{\mathcal{X}^0}^3dz+\alpha\int_0^t\|u(z)\|_{\mathcal{X}^0}^{2m+1}dz\\
&\leq &\|u^0\|_{\mathcal{X}^0}+ \int_0^t\left(\frac{1}{2}\|u(z)\|_{\mathcal{X}^0}^2+\alpha \|u(z)\|_{\mathcal{X}^0}^{2m}\right)\|u(z)\|_{\mathcal{X}^0}dz.\\
\end{eqnarray*}
By Gr\"onwall Lemma, we obtain estimate \eqref{theo2-eqn1}.
\section{\bf Proof of Theorem \ref{theo3}}
Let $R$ be a positive real number. For $T>0$, we define the Banach space
$$
E_T^{'}=\mathcal{C}_{T}\bigl(\mathcal{X}_{a,\sigma}^{0}(\mathbb{R}^{3})\bigr)
\cap L^{1}_{T}\bigl(\mathcal{X}_{a,\sigma}^{2}(\mathbb{R}^{3})\bigr)
$$
equipped with norm
$$
\|u\|_{E_T^{'}}=\|u\|_{L_T^{\infty}(\mathcal{X}_{a,\sigma}^0)}+\|u\|_{L_T^1(\mathcal{X}_{a,\sigma}^2)}.
$$
We denote the closed subset $F_{R, T}^{'}$ of $E_T^{'}$ defined by
$$
F_{R, T}^{'}=\Big\{u\in E_T^{'}; \|u\|_{L_T^{\infty}(\mathcal{X}_{a,\sigma}^0)}\leq 2\|u^0\|_{\mathcal{X}_{a,\sigma}^0},\;\|u\|_{L_T^1(\mathcal{X}_{a,\sigma}^2)}\leq R\Big\}.
$$
Now, consider the operator
\begin{eqnarray*}
\Phi\,&:&\,F_{R, T}^{'}\longrightarrow	 E_T^{'}\\
               && u\longmapsto \Phi (u)= e^{t\Delta}u^0-\mathcal{B}(u,u)(t)-\alpha \mathcal{D}_m(u)(t).
\end{eqnarray*}
To apply the fixed-point theorem, it suffices to prove that the following estimates hold
for $T$ and $R$ sufficiently small:
\begin{eqnarray}\label{eq51}
&&\Phi(F_{R, T}^{'})\subset F_{R, T}^{'}\\\label{eq52}
&& \|\Phi(u)-\Phi(v)\|_{E_T^{'}}\leq \frac{1}{2}\|u-v\|_{E_T^{'}},\,\,\forall u,v\in F_{R,T}^{'}.
\end{eqnarray}
 First, we prove \eqref{eq51}: \pn
$\bullet$ To estimate $\Phi(u)$ in $\mathcal{X}_{a,\sigma}^0(\mathbb{R}^3)$, we have
\begin{eqnarray*}
\| e^{t\Delta}u^0\|_{L_T^{\infty}(\mathcal{X}_{a,\sigma}^0)} &\leq & \|u^0\|_{\mathcal{X}_{a,\sigma}^0}.\hskip 3cm
\end{eqnarray*}
 Using inequalities \eqref{lem44-eqn1}-\eqref{lem55-eqn3} and the fact that $u\in F_{R,T}^{'}$, we obtain
\begin{eqnarray*}
\|\mathcal{B}(u,u)\|_{L_T^{\infty}(\mathcal{X}_{a,\sigma}^0)}&\leq &T^{\frac{1}{2}}\|u\|_{L_T^{\infty}(\mathcal{X}_{a,\sigma}^0)}^{\frac{3}{2}}\|u\|_{L_T^{1}(\mathcal{X}_{a,\sigma}^2)}^{\frac{1}{2}} \\
&\leq &  2\sqrt{2}T^{\frac{1}{2}}R^{\frac{1}{2}}\|u^0\|_{\mathcal{X}_{a,\sigma}^0}^{\frac{3}{2}}
\end{eqnarray*}
and
\begin{eqnarray*}
\| \mathcal{D}_m(u)\|_{L_T^{\infty}(\mathcal{X}_{a,\sigma}^0)}
&\leq & T\|u\|_{L_T^{\infty}(\mathcal{X}_{a,\sigma}^0)}^{2m+1}\\
&\leq & 2^{2m+1}  T\|u^0\|_{\mathcal{X}_{a,\sigma}^0}^{2m+1}.
\end{eqnarray*}
Then, there is a time $T_1>0$ such that
$$
  2\sqrt{2}T^{\frac{1}{2}}R^{\frac{1}{2}}\|u^0\|_{\mathcal{X}_{a,\sigma}^0}^{\frac{3}{2}}+ 2^{2m+1}\alpha T\|u^0\|_{\mathcal{X}_{a,\sigma}^0}^{2m+1} \leq \|
u^0\|_{\mathcal{X}_{a,\sigma}^0},\,\,\forall \,0<T\leq T_1.
$$
By using this choice, we obtain
\begin{eqnarray}\label{eq53}
\|\Phi(u)\|_{L_T^{\infty}(\mathcal{X}_{a,\sigma}^0)}\leq 2\|u^0\|_{\mathcal{X}_{a,\sigma}^0},\,\forall u\in 	F_{R, T}^{'}.
\end{eqnarray}
\noindent $\bullet$ To estimate $\Phi(u)$ in $L^1_T(\mathcal{X}_{a,\sigma}^2(\mathbb{R}^3))$, we have
\begin{eqnarray*}
\|e^{t\Delta}u^0\|_{L_T^{1}(\mathcal{X}_{a,\sigma}^2)} &= & \int_0^T\int_{\xi}|\xi|^2e^{-t|\xi|^2} e^{a|\xi|^{\frac{1}{\sigma}}}|\wh{u^0(\xi)}|d\xi dt\\
&=&\int_{\xi}\left(1-e^{-T|\xi|^2}\right) e^{a|\xi|^{\frac{1}{\sigma}}}|\wh{u^0(\xi)}|d\xi.
\end{eqnarray*}
By dominated convergence theorem, there is a time $T_2>0$ such that
$$
\|e^{t\Delta}u^0\|_{L_T^{1}(\mathcal{X}_{a,\sigma}^2)}<\frac{R}{3},\;\forall\,0<T\leq T_2.
$$
By the inequality \eqref{lem44-eqn2}, we get
\begin{eqnarray*}
\|\mathcal{B}(u,u)\|_{L_T^{1}(\mathcal{X}_{a,\sigma}^2)} &\leq &T^{\frac{1}{2}}\|u\|_{L_T^{\infty}(\mathcal{X}_{a,\sigma}^0)}^{\frac{3}{2}}\|u\|_{L_T^{1}(\mathcal{X}_{a,\sigma}^2)}^{\frac{1}{2}}\\
&\leq &2\sqrt{2}T^{\frac{1}{2}} R^{\frac{1}{2}}\|u^0\|_{\mathcal{X}_{a,\sigma}^0}^{\frac{3}{2}} .
\end{eqnarray*}
By inequality \eqref{lem55-eqn4}, we have
\begin{eqnarray*}
\|\mathcal{D}_m(u)\|_{L_T^{1}(\mathcal{X}_{a,\sigma}^2)}&\leq& T\|u\|_{L_T^{\infty}(\mathcal{X}_{a,\sigma}^0)}^{2m+1}\\
&\leq& 2^{2m+1} T\|u^0\|_{\mathcal{X}_{a,\sigma}^0}^{2m+1}.
\end{eqnarray*}
Therefore, there exists a time $T_3>0$ such that
$$
2\sqrt{2}T^{\frac12}R^{\frac12}
\|u^0\|_{\mathcal{X}_{a,\sigma}^0}^{\frac32}
+
2^{2m+1}\alpha T
\|u^0\|_{\mathcal{X}_{a,\sigma}^0}^{2m+1}
\leq \frac{2R}{3},
\qquad \forall\,0<T\leq T_3.
$$
Hence,
\begin{equation}\label{eq54}
\|\Phi(u)\|_{L_T^1(\mathcal{X}_{a,\sigma}^2)}
\leq  R,
\qquad \forall u\in F_{R,T}'.
\end{equation}
Then, choosing
$$
0<T\leq \min\{T_1,T_2,T_3\},
$$
we deduce from \eqref{eq53} and \eqref{eq54} that
$$
\Phi(F_{R,T}')\subset F_{R,T}'.
$$
Second, we prove \eqref{eq52}:
By using \eqref{blow-10}-\eqref{lem44-eqn1}, for $u,v\in F_{R, T}^{'}$
\begin{eqnarray}\nonumber
\|\mathcal{B}(u,u)-\mathcal{B}(v,v)\|_{L_T^{\infty}(\mathcal{X}_{a,\sigma}^0)}
&\leq & T^{\frac{1}{2}}\,\|u\|_{L_T^{\infty}(\mathcal{X}_{a,\sigma}^0)}
\|u-v\|_{L_T^{\infty}(\mathcal{X}_{a,\sigma}^0)}^{\frac{1}{2}}\|u-v\|_{L_T^{1}(\mathcal{X}_{a,\sigma}^2)}^{\frac{1}{2}}\\\nonumber
&& + T^{\frac{1}{2}}\,\|u-v\|_{L_T^\infty(\mathcal{X}_{a,\sigma}^0)}\|v\|_{L_T^{\infty}(\mathcal{X}_{a,\sigma}^0)}^{\frac{1}{2}}\|v\|_{L_T^{1}(\mathcal{X}_{a,\sigma}^2)}^{\frac{1}{2}}\\\nonumber
&\leq& 2 T^{\frac{1}{2}}\,\|u^0\|_{\mathcal{X}_{a,\sigma}^0}\|u-v\|_{E_T^{'}}+\sqrt{2} T^{\frac{1}{2}}R^{\frac{1}{2}}\|u^0\|_{\mathcal{X}_{a,\sigma}^0}^{\frac{1}{2}} \|u-v\|_{E_T^{'}}\\\label{eq55}
&\leq &\sqrt{2}T^{\frac{1}{2}}\|u^0\|_{\mathcal{X}_{a,\sigma}^0}^{\frac{1}{2}}\left(\sqrt{2}\|u^0\|_{\mathcal{X}_{a,\sigma}^0}^{\frac{1}{2}}+R^{\frac{1}{2}}\right)\|u-v\|_{E_T^{'}}.
\end{eqnarray}
By using inequality \eqref{lem44-eqn2} , we obtain
\begin{eqnarray}\nonumber
\|\mathcal{B}(u,u)-\mathcal{B}(v,v)\|_{L_T^1(\mathcal{X}_{a,\sigma}^2)}	
&\leq & T^{\frac{1}{2}}\,\|u\|_{L_T^{\infty}(\mathcal{X}_{a,\sigma}^0)}
\|u-v\|_{L_T^{\infty}(\mathcal{X}_{a,\sigma}^0)}^{\frac{1}{2}}\|u-v\|_{L_T^{1}(\mathcal{X}_{a,\sigma}^2)}^{\frac{1}{2}}\\\nonumber
&& + T^{\frac{1}{2}}\,\|u-v\|_{L_T^\infty(\mathcal{X}_{a,\sigma}^0)}\|v\|_{L_T^{\infty}(\mathcal{X}_{a,\sigma}^0)}^{\frac{1}{2}}\|v\|_{L_T^{1}
(\mathcal{X}_{a,\sigma}^2)}^{\frac{1}{2}}\\\nonumber
&\leq & 2 T^{\frac{1}{2}}\,\|u^0\|_{\mathcal{X}_{a,\sigma}^0}\|u-v\|_{E_T^{'}}+\sqrt{2}T^{\frac{1}{2}} R^{\frac{1}{2}}\|u^0\|_{\mathcal{X}_{a,\sigma}^0}^{\frac{1}{2}} \|u-v\|_{E_T^{'}}\\\label{eq56}
&\leq &\sqrt{2} T^{\frac{1}{2}}\|u^0\|_{\mathcal{X}_{a,\sigma}^0}^{\frac{1}{2}}\left(\sqrt{2}\|u^0\|_{\mathcal{X}_{a,\sigma}^0}^{\frac{1}{2}}+R^{\frac{1}{2}}\right)\|u-v\|_{E_T^{'}}.
\end{eqnarray}
Now, using  \eqref{blow-11}-\eqref{blow-12}-\eqref{lem55-eqn1}, we get
\begin{eqnarray}\nonumber
\|\mathcal{D}_m(u)-\mathcal{D}_m(v)\|_{L_T^{\infty}(\mathcal{X}_{a,\sigma}^0)}
&\leq & T\Big[\|u\|_{L_T^{\infty}(\mathcal{X}_{a,\sigma}^0)}^{2m}+ \|u\|_{L_T^{\infty}(\mathcal{X}_{a,\sigma}^0)}^{2m-1}\|v\|_{L_T^{\infty}(\mathcal{X}_{a,\sigma}^0)}+..\\\nonumber
&&..+ \|u\|_{L_T^{\infty}(\mathcal{X}_{a,\sigma}^0)}\|v\|_{L_T^{\infty}(\mathcal{X}_{a,\sigma}^0)}^{2m-1}+\|v\|_{L_T^{\infty}(\mathcal{X}_{a,\sigma}^0)}^{2m}\Big]\|u-v\|_{L_T^{\infty}(\mathcal{X}_{a,\sigma}^0)}\\\nonumber
&\leq & (2m+1)T\Big[\|u\|_{L_T^{\infty}(\mathcal{X}_{a,\sigma}^0)}^{2m}+  \|v\|_{L_T^{\infty}(\mathcal{X}_{a,\sigma}^0)}^{2m}\Big]\|u-v\|_{L_T^{\infty}(\mathcal{X}_{a,\sigma}^0)}\\\label{eq57}
&\leq & (2m+1)2^{2m+1}T\|u^0\|_{\mathcal{X}_{a,\sigma}^0}^{2m}\|u-v\|_{E_T^{'}}.
\end{eqnarray}
By using inequalities \eqref{lem55-eqn2}, we get
\begin{eqnarray}\nonumber
\|\mathcal{D}_m(u)-\mathcal{D}_m(v)\|_{L_T^1(\mathcal{X}_{a,\sigma}^2)}
&\leq &  T\Big[\|u\|_{L_T^{\infty}(\mathcal{X}_{a,\sigma}^0)}^{2m}+ \|u\|_{L_T^{\infty}(\mathcal{X}_{a,\sigma}^0)}^{2m-1}\|v\|_{L_T^{\infty}(\mathcal{X}_{a,\sigma}^0)}+...\\\nonumber
&&...+ \|u\|_{L_T^{\infty}(\mathcal{X}_{a,\sigma}^0)}\|v\|_{L_T^{\infty}(\mathcal{X}_{a,\sigma}^0)}^{2m-1}+\|v\|_{L_T^{\infty}(\mathcal{X}_{a,\sigma}^0)}^{2m}\Big]\|u-v\|_{L_T^{\infty}(\mathcal{X}_{a,\sigma}^0)}\\\nonumber
&\le & (2m+1)T\Big[\|u\|_{L_T^{\infty}
(\mathcal{X}_{a,\sigma}^0)}^{2m}+\|v\|_{L_T^{\infty}
(\mathcal{X}_{a,\sigma}^0)}^{2m}\Big]\|u-v\|_{L_T^{\infty}(\mathcal{X}_{a,\sigma}^0)}\\\label{eq58}
&\leq & (2m+1)2^{2m+1}T\|u^0\|_{\mathcal{X}_{a,\sigma}^0}^{2m}\|u-v\|_{E_T^{'}}
\end{eqnarray}
Finally, by decreasing $T>0$ if necessary, while keeping
$$
T\leq \min\{T_1,T_2,T_3\},
$$
we may assume that
\begin{eqnarray*}
&&
\sqrt{2}T^{\frac12}
\|u^0\|_{\mathcal{X}_{a,\sigma}^0}^{\frac12}
\left(
\sqrt{2}\|u^0\|_{\mathcal{X}_{a,\sigma}^0}^{\frac12}
+R^{\frac12}
\right)
\leq \frac18,
\\
&&
(2m+1)2^{2m+1}\alpha T
\|u^0\|_{\mathcal{X}_{a,\sigma}^0}^{2m}
\leq \frac18.
\end{eqnarray*}
Combining \eqref{eq55}--\eqref{eq58}, we obtain
$$
\|\Phi(u)-\Phi(v)\|_{E_T'}\leq \frac12\|u-v\|_{E_T'}.
$$
Hence, $\Phi$ is a contraction on $F_{R,T}'$. Therefore, the Banach fixed-point theorem yields
a unique local solution to $(NSE_M)$ in
$$
\mathcal{C}_{T}
\bigl(\mathcal{X}_{a,\sigma}^{0}(\mathbb{R}^{3})\bigr)
\cap
L^{1}_{T}
\bigl(\mathcal{X}_{a,\sigma}^{2}(\mathbb{R}^{3})\bigr).
$$
This completes the proof of Theorem~\ref{theo3}.
\section{\bf Proof of Theorem \ref{theo4}}
\subsection{Proof of \eqref{theo4-eqn1}}\pn
Let $u\in \mathcal{C}([0,T^*),\X_{a,\sigma}^{0}(\mathbb{R}^{3}))\cap L^{1}_{loc}([0,T^*),\X_{a,
\sigma}^{2}(\mathbb{R}^{3}))$ be the maximal solution of problem $(NSE_M)$ given by Theorem
\ref{theo3}. From the first equation of the system $(NSE_M)$, one has
\begin{eqnarray*}
&& \|u(t)\|_{\mathcal{X}_{a,\sigma}^0}+\int_0^t\|u(z)\|_{\mathcal{X}_{a,\sigma}^2}dz\\
&& \quad \leq  \|u^0\|_{\mathcal{X}_{a,\sigma}^0}+\int_0^t\|u(z)\|_{\mathcal{X}_{a,\sigma}
 ^0}\|u(z)\|_{\mathcal{X}_{a,\sigma}^1}dz+\alpha\int_0^t\|u(z)\|_{\mathcal{X}_{a,\sigma}^0}
 ^{2m+1}dz.	
\end{eqnarray*}
By using inequality \eqref{lem1-eqn2}, we get
\begin{eqnarray*}
&& \|u(t)\|_{\mathcal{X}_{a,\sigma}^0}+\int_0^t\|u(z)\|_{\mathcal{X}_{a,\sigma}^2}dz\\
&& \quad \leq  \|u^0\|_{\mathcal{X}_{a,\sigma}^0}+\int_0^t\|u(z)\|_{\mathcal{X}_{a,\sigma}^0}^{\frac{3}{2}}\|u(z)\|_{\mathcal{X}_{a,\sigma}^2}^{\frac{1}{2}}dz+\alpha\int_0^t\|u(z)\|_{\mathcal{X}_{a,\sigma}^0}^{2m+1}dz\\	
\end{eqnarray*}
The elementary inequality
$$
xy\leq \frac{x^2}{2}+\frac{y^2}{2},\;\forall x,y\geq0,
$$
implies that
\begin{eqnarray*}
\|u(t)\|_{\mathcal{X}_{a,\sigma}^0}+\frac{1}{2}\int_0^t\|u(z)\|_{\mathcal{X}_{a,\sigma}^2}dz
&\leq &\|u^0\|_{\mathcal{X}_{a,\sigma}^0}+\frac{1}{2}\int_0^t\|u(z)\|_{\mathcal{X}_{a,\sigma}^0}^3dz+\alpha\int_0^t\|u(z)\|_{\mathcal{X}_{a,\sigma}^0}^{2m+1}dz\\
&\leq &\|u^0\|_{\mathcal{X}_{a,\sigma}^0}+ \int_0^t\left(\frac{1}{2}\|u(z)\|_{\mathcal{X}_{a,\sigma}^0}^2+\alpha \|u(z)\|_{\mathcal{X}_{a,\sigma}^0}^{2m}\right)\|u(z)\|_{\mathcal{X}_{a,\sigma}^0}dz.\\
\end{eqnarray*}
By Gr\"onwall Lemma, we obtain estimate \eqref{theo4-eqn1}.
\subsection{Proof of \eqref{theo4-eqn2}}\pn
We prove that
$$
\limsup_{t\rightarrow T^*}
\|u(t)\|_{\mathcal{X}_{a,\sigma}^0}
=+\infty.
$$
Suppose, by contradiction, that $u$ is bounded in $\mathcal{X}_{a,\sigma}^0(\mathbb{R}^3)$, and
set
\begin{equation}\label{prooftheo4-eq30}
M_*=\sup_{t\in[0,T^*)}
\|u(t)\|_{\mathcal{X}_{a,\sigma}^0}<\infty.
\end{equation}
Firstly, applying Fourier transform and taking the scalar product in $\mathbb{R}^3$ of the first equation with $ \widehat{u}(t)=\mathcal{F}\left(\sum_{k=1}^3u_ke_k\right)$, we obtain
\begin{eqnarray}\nonumber
&&\frac{1}{2}\partial_t|\widehat{u}(t,\xi)|^2+|\xi|^2|\widehat{u}(t,\xi)|^2\\\label{prooftheo4-eq31}
&&\quad +\mathbin{Re}\left(\mathcal{F}(u\cdot\nabla u)\cdot \widehat{u}(t,-\xi)\right)+\alpha \left(\mathcal{F}\left(\sum_{k=1}^3u_k^{2m+1}e_k\right)\cdot \mathcal{F}\left(\sum_{k=1}^3u_ke_k\right)(t,-\xi)\right)=0
\end{eqnarray}
For any arbitrary $\varepsilon >0$
\begin{eqnarray}\nonumber
\frac{1}{2}\partial_t|\widehat{u}(t,\xi)|^2&=&\frac{1}{2}\partial_t(|\widehat{u}(t,\xi)|^2+\varepsilon)\\\label{prooftheo4-eq32}
 &=&  \sqrt{(|\widehat{u}(t,\xi)|^2+\varepsilon)	}\cdot  \partial_t\sqrt{|\widehat{u}(t,\xi)|^2+\varepsilon}.
\end{eqnarray}
Substituting \eqref{prooftheo4-eq32} in  \eqref{prooftheo4-eq31}, we get
\begin{eqnarray*}
&&\partial_t\sqrt{(|\widehat{u}(t,\xi)|^2+\varepsilon)	}+|\xi|^2\frac{|\widehat{u}(t,\xi)|^2}{\sqrt{|\widehat{u}(t,\xi)|^2+\varepsilon}}	\\
&&\quad + \frac{\mathbin{Re}	\left(\mathcal{F}(u\cdot\nabla u)\cdot \widehat{u}(t,-\xi)\right)}{\sqrt{|\widehat{u}(t,\xi)|^2+\varepsilon}}+ \alpha\frac{  \left(\mathcal{F}\left(\sum_{k=1}^3u_k^{2m+1}e_k\right)\cdot \mathcal{F}\left(\sum_{k=1}^3u_ke_k\right)(t,-\xi)\right)}{\sqrt{|\widehat{u}(t,\xi)|^2+\varepsilon}}=0
\end{eqnarray*}
This implies
\begin{eqnarray*}
&&\partial_t\sqrt{(|\widehat{u}(t,\xi)|^2+\varepsilon)	}+|\xi|^2\frac{|\widehat{u}(t,\xi)|^2}{\sqrt{|\widehat{u}(t,\xi)|^2+\varepsilon}}	\\
&&\quad \leq  |\mathcal{F}(u\cdot\nabla u)(t,\xi)| +\alpha |\mathcal{F}(\sum_{k=1}^3u_k^{2m+1}e_k)(t,\xi)|.
\end{eqnarray*}
Integrating with respect to time and letting $\varepsilon\to0^+$, we obtain
$$
|\widehat{u}(t,\xi)|
+\int_0^t |\xi|^2|\widehat{u}(\tau,\xi)|\,d\tau
\leq
|\widehat{u^0}(\xi)|
+\int_0^t
|\mathcal{F}(u\cdot\nabla u)(\tau,\xi)|\,d\tau
+\alpha\int_0^t
\left|
\mathcal{F}\left(
\sum_{k=1}^3u_k^{2m+1}e_k
\right)(\tau,\xi)
\right|\,d\tau.
$$
Multiplying by $e^{a|\xi|^{\frac{1}{\sigma}}}$, we obtain
\begin{eqnarray}\nonumber
&&  \|u(t)\|_{\mathcal{X}_{a,\sigma}^0}+\int_0^t \|\Delta u(\tau)\|_{\mathcal{X}_{a,\sigma}^0}d\tau\\\label{prooftheo4-eq33}
&&\quad \leq   \|u^0\|_{\mathcal{X}_{a,\sigma}^0}+\int_0^t\|u\cdot\nabla u\|_{\mathcal{X}_{a,\sigma}^0}d\tau  +\alpha \int_0^t \|u^{2m+1}\|_{\mathcal{X}_{a,\sigma}^0}d\tau.
\end{eqnarray}
Using inequalities \eqref{lem1-eqn2}-\eqref{lem33-eqn2}, we get
\begin{eqnarray*}
&&  \|u(t)\|_{\mathcal{X}_{a,\sigma}^0}+\int_0^t \|\Delta u(\tau)\|_{\mathcal{X}_{a,\sigma}^0}d\tau\\
&&\quad \leq   \|u^0\|_{\mathcal{X}_{a,\sigma}^0}+\int_0^t\|u\cdot\nabla u\|_{\mathcal{X}_{a,\sigma}^0}d\tau  +\alpha \int_0^t \|u^{2m+1}\|_{\mathcal{X}_{a,\sigma}^0}d\tau\\
&&\quad \leq   \|u^0\|_{\mathcal{X}_{a,\sigma}^0}+\int_0^t\|u\|_{\mathcal{X}_{a,\sigma}^0}
\|\nabla u\|_{\mathcal{X}_{a,\sigma}^0} d\tau  +\alpha \int_0^t \|u^{2m+1}\|_{\mathcal{X}_{a,\sigma}^0}d\tau\\
&&\quad \leq   \|u^0\|_{\mathcal{X}_{a,\sigma}^0}+\int_0^t\|u\|_{\mathcal{X}_{a,\sigma}^0}^{\frac{3}{2}}
\|\Delta  u\|_{\mathcal{X}_{a,\sigma}^0}^{\frac{1}{2}} d\tau  +2^{2m}\alpha \int_0^t \|u \|_{\mathcal{X}_{a,\sigma}^0}^{2m+1}d\tau\\
\end{eqnarray*}
By Young's inequality, we obtain
\begin{eqnarray*}
&&  \|u(t)\|_{\mathcal{X}_{a,\sigma}^0}
+\frac{1}{2}\int_0^t
\|\Delta u(\tau)\|_{\mathcal{X}_{a,\sigma}^0}\,d\tau
\\
&&\quad \leq
\|u^0\|_{\mathcal{X}_{a,\sigma}^0}
+ \int_0^t
\left(
\frac{1}{2}\|u(\tau)\|_{\mathcal{X}_{a,\sigma}^0}^3
+2^{2m}\alpha
\|u(\tau)\|_{\mathcal{X}_{a,\sigma}^0}^{2m+1}
\right)d\tau .
\end{eqnarray*}
Using \eqref{prooftheo4-eq30}, we deduce that, for every \(t<T^*\),
$$
\frac{1}{2}\int_0^t
\|\Delta u(\tau)\|_{\mathcal{X}_{a,\sigma}^0}\,d\tau
\leq
\|u^0\|_{\mathcal{X}_{a,\sigma}^0}
+
\left(
\frac{1}{2}M_*^3
+2^{2m}\alpha M_*^{2m+1}
\right)t .
$$
Hence,
\begin{eqnarray}\label{prooftheo4-eq34}
\int_0^{T^*}
\|\Delta u(\tau)\|_{\mathcal{X}_{a,\sigma}^0}\,d\tau
&\leq&
2\|u^0\|_{\mathcal{X}_{a,\sigma}^0}
+
\left(
M_*^3
+2^{2m+1}\alpha M_*^{2m+1}
\right)T^*
=:C_* .
\end{eqnarray}
Secondly, let $(t_p)_{p\geq1}\subset [0,T^*)$
be such that
$$
\lim_{p\rightarrow\infty}t_p=T^*.
$$
We prove that $u$ satisfies the Cauchy property at $T^*$.
Using Duhamel's formula, for $p,q\in\mathbb{N}$ with $t_p<t_q$, we have
$$
\begin{array}{lcl}
u(t_p)&=&e^{t_p\Delta}u^0-\mathcal{B}(u,u)(t_p)-\alpha\mathcal{D}_m(u)(t_p),\\
u(t_q)&=&e^{t_q\Delta}u^0-\mathcal{B}(u,u)(t_q)-\alpha\mathcal{D}_m(u)(t_q).
\end{array}
$$
Therefore,
$$
u(t_p)-u(t_q)=\sum_{k=1}^5 I_{p,q}^{(k)},
$$
where
$$
\begin{array}{lcl}
I_{p,q}^{(1)}
&=&
(e^{t_p\Delta}-e^{t_q\Delta})u^0,
\\[1mm]
I_{p,q}^{(2)}
&=&
\displaystyle\int_0^{t_p}
\big(e^{(t_q-z)\Delta}-e^{(t_p-z)\Delta}\big)
\mathbb{P}(u\cdot \nabla u)(z)\,dz,
\\[2mm]
I_{p,q}^{(3)}
&=&
\displaystyle \int_{t_p}^{t_q}
e^{(t_q-z)\Delta}\mathbb{P}(u\cdot \nabla u)(z)\,dz,
\\[2mm]
I_{p,q}^{(4)}
&=&
\displaystyle\alpha \int_0^{t_p}
\big(e^{(t_q-z)\Delta}-e^{(t_p-z)\Delta}\big)
\mathbb{P}\left(\sum_{k=1}^3u_k^{2m+1}e_k\right)(z)\,dz,
\\[2mm]
I_{p,q}^{(5)}
&=&
\displaystyle\alpha \int_{t_p}^{t_q}
e^{(t_q-z)\Delta}
\mathbb{P}\left(\sum_{k=1}^3u_k^{2m+1}e_k\right)(z)\,dz.
\end{array}
$$
\begin{enumerate}
\item[$\bullet$] Study of $I_{p,q}^{(1)}$ in $\mathcal{X}_{a,\sigma}^0(\mathbb{R}^3)$. We have
$$
\begin{array}{lcl}
\|I_{p,q}^{(1)}\|_{\mathcal{X}_{a,\sigma}^0}
&=&
\displaystyle\int_{\mathbb{R}^3}
\left|e^{-t_p|\xi|^2}-e^{-t_q|\xi|^2}\right|
e^{a|\xi|^{\frac{1}{\sigma}}}
|\widehat{u^0}(\xi)|\,d\xi
\\[2mm]
&=&
\displaystyle\int_{\mathbb{R}^3}
e^{-t_p|\xi|^2}
\left|1-e^{-(t_q-t_p)|\xi|^2}\right|
e^{a|\xi|^{\frac{1}{\sigma}}}
|\widehat{u^0}(\xi)|\,d\xi
\\[2mm]
&\leq&
\displaystyle\int_{\mathbb{R}^3}
\left|1-e^{-(t_q-t_p)|\xi|^2}\right|
e^{a|\xi|^{\frac{1}{\sigma}}}
|\widehat{u^0}(\xi)|\,d\xi.
\end{array}
$$
Since $t_q-t_p\to0^+$ as $p,q\to\infty$, the dominated convergence theorem gives
$$
\lim_{p,q\to\infty}
\|I_{p,q}^{(1)}\|_{\mathcal{X}_{a,\sigma}^0}=0.
$$
\item[$\bullet$] Study of $I_{p,q}^{(2)}$ in $\mathcal{X}_{a,\sigma}^0(\mathbb{R}^3)$. Using the fact
that $\operatorname{div}u=0$, we write
$$
I_{p,q}^{(2)}
=
-\int_0^{t_p}
\big(e^{(t_q-z)\Delta}-e^{(t_p-z)\Delta}\big)
\mathbb{P}\operatorname{div}(u\otimes u)(z)\,dz.
$$
Then
$$
\begin{array}{lcl}
\|I_{p,q}^{(2)}\|_{\mathcal{X}_{a,\sigma}^0}
&\leq&
\displaystyle
\int_{\mathbb{R}^3}\int_0^{t_p}
\left|e^{-(t_q-z)|\xi|^2}-e^{-(t_p-z)|\xi|^2}\right|
e^{a|\xi|^{\frac{1}{\sigma}}}
|\mathcal F(u\cdot\nabla u)(z,\xi)|\,dz\,d\xi
\\[2mm]
&=&
\displaystyle
\int_{\mathbb{R}^3}\int_0^{t_p}
\left|e^{-(t_q-t_p)|\xi|^2}-1\right|
e^{-(t_p-z)|\xi|^2}
e^{a|\xi|^{\frac{1}{\sigma}}}
|\mathcal F(u\cdot\nabla u)(z,\xi)|\,dz\,d\xi
\\[2mm]
&\leq&
\displaystyle
\int_0^{T^*}\int_{\mathbb{R}^3}
\left|e^{-(t_q-t_p)|\xi|^2}-1\right|
e^{a|\xi|^{\frac{1}{\sigma}}}
|\mathcal F(u\cdot\nabla u)(z,\xi)|\,d\xi\,dz .
\end{array}
$$
Since $t_q-t_p\to0^+$ as $p,q\to\infty$, and
$$
\left|e^{-(t_q-t_p)|\xi|^2}-1\right|
e^{a|\xi|^{\frac{1}{\sigma}}}
|\mathcal F(u\cdot\nabla u)(z,\xi)|
\leq
e^{a|\xi|^{\frac{1}{\sigma}}}
|\mathcal F(u\cdot\nabla u)(z,\xi)|,
$$
Moreover, by \eqref{lem1-eqn2}, \eqref{prooftheo4-eq30},
and \eqref{prooftheo4-eq34},
$$
\begin{aligned}
\int_0^{T^*}
\|u\cdot\nabla u\|_{\mathcal{X}_{a,\sigma}^0}\,dz
&\leq
\int_0^{T^*}
\|u(z)\|_{\mathcal{X}_{a,\sigma}^0}^{\frac32}
\|u(z)\|_{\mathcal{X}_{a,\sigma}^2}^{\frac12}\,dz
\\
&\leq
M_*^{\frac32}(T^*)^{\frac12}C_*^{\frac12}
<\infty.
\end{aligned}
$$
the dominated convergence theorem yields
$$
\lim_{p,q\to\infty}
\|I_{p,q}^{(2)}\|_{\mathcal{X}_{a,\sigma}^0}=0.
$$
\item[$\bullet$] Study of $I_{p,q}^{(3)}$ in $\mathcal{X}_{a,\sigma}^0(\mathbb{R}^3)$. By inequalities \eqref{lem1-eqn2}-\eqref{prooftheo4-eq30}-\eqref{prooftheo4-eq34}, we have
$$
\begin{array}{lcl}
\|I_{p,q}^{(3)}\|_{\mathcal{X}_{a,\sigma}^0}
&\leq&
\displaystyle
\int_{\mathbb{R}^3}\int_{t_p}^{t_q}
e^{-(t_q-z)|\xi|^2}
e^{a|\xi|^{\frac{1}{\sigma}}}
|\mathcal F(u\cdot \nabla u)(z,\xi)|\,dz\,d\xi
\\[2mm]
&\leq&
\displaystyle
\int_{t_p}^{t_q}
\int_{\mathbb{R}^3}
e^{a|\xi|^{\frac{1}{\sigma}}}
|\mathcal F(u\cdot \nabla u)(z,\xi)|\,d\xi\,dz
\\[2mm]
&\leq&
\displaystyle
\int_{t_p}^{t_q}
\|(u\otimes u)(z)\|_{\mathcal{X}_{a,\sigma}^{1}}\,dz
\\[2mm]
&\leq&
\displaystyle
\int_{t_p}^{t_q}
\|u(z)\|_{\mathcal{X}_{a,\sigma}^{0}}
\|u(z)\|_{\mathcal{X}_{a,\sigma}^{1}}\,dz
\\[2mm]
&\leq&
\displaystyle
\int_{t_p}^{T^*}
\|u(z)\|_{\mathcal{X}_{a,\sigma}^{0}}
\|u(z)\|_{\mathcal{X}_{a,\sigma}^{1}}\,dz
\\[2mm]
&\leq&
\displaystyle
\int_{t_p}^{T^*}
\|u(z)\|_{\mathcal{X}_{a,\sigma}^{0}}^{\frac32}
\|u(z)\|_{\mathcal{X}_{a,\sigma}^{2}}^{\frac12}\,dz
\\[2mm]
&\leq&
M_*^{\frac32}
\int_{t_p}^{T^*}
\|u(z)\|_{\mathcal{X}_{a,\sigma}^{2}}^{\frac12}\,dz
\\[2mm]
&\leq&
M_*^{\frac32}\sqrt{C_*(T^*-t_p)}.
\end{array}
$$
Since $t_p\to T^*$ as $p\to\infty$, it follows that
$$
\lim_{p,q\to\infty}
\|I_{p,q}^{(3)}\|_{\mathcal{X}_{a,\sigma}^{0}}
=0.
$$
\item[$\bullet$] Study of $I_{p,q}^{(4)}$ in $\mathcal{X}_{a,\sigma}^0(\mathbb{R}^3)$. We have
$$
\begin{array}{lcl}
\|I_{p,q}^{(4)}\|_{\mathcal{X}_{a,\sigma}^0}
&\leq&
\displaystyle \alpha \int_{\mathbb{R}^3}\int_0^{t_p}
\left|e^{-(t_q-z)|\xi|^2}-e^{-(t_p-z)|\xi|^2}\right|
e^{a|\xi|^{\frac{1}{\sigma}}}
\left|\mathcal{F}\left(\sum_{k=1}^3u_k^{2m+1}e_k\right)(z,\xi)\right|
\,dz\,d\xi
\\[2mm]
&=&
\displaystyle \alpha \int_{\mathbb{R}^3}\int_0^{t_p}
\left|e^{-(t_q-t_p)|\xi|^2}-1\right|
e^{-(t_p-z)|\xi|^2}
e^{a|\xi|^{\frac{1}{\sigma}}}
\left|\mathcal{F}\left(\sum_{k=1}^3u_k^{2m+1}e_k\right)(z,\xi)\right|
\,dz\,d\xi
\\[2mm]
&\leq&
\displaystyle \alpha \int_0^{T^*}\int_{\mathbb{R}^3}
\left|e^{-(t_q-t_p)|\xi|^2}-1\right|
e^{a|\xi|^{\frac{1}{\sigma}}}
\left|\mathcal{F}\left(\sum_{k=1}^3u_k^{2m+1}e_k\right)(z,\xi)\right|
\,d\xi\,dz .
\end{array}
$$
Since $t_q-t_p\to0^+$ as $p,q\to\infty$, and
$$
\left|e^{-(t_q-t_p)|\xi|^2}-1\right|
e^{a|\xi|^{\frac{1}{\sigma}}}
\left|\mathcal{F}\left(\sum_{k=1}^3u_k^{2m+1}e_k\right)(z,\xi)\right|
\leq
e^{a|\xi|^{\frac{1}{\sigma}}}
\left|\mathcal{F}\left(\sum_{k=1}^3u_k^{2m+1}e_k\right)(z,\xi)\right|,
$$
and
$$
\sum_{k=1}^3u_k^{2m+1}e_k
\in
L^1\big([0,T^*);\mathcal{X}_{a,\sigma}^0(\mathbb{R}^3)\big),
$$
the dominated convergence theorem yields
$$
\lim_{p,q\to\infty}
\|I_{p,q}^{(4)}\|_{\mathcal{X}_{a,\sigma}^0}=0.
$$
\item[$\bullet$] Study of $I_{p,q}^{(5)}$ in $\mathcal{X}_{a,\sigma}^0(\mathbb{R}^3)$. By inequalities \eqref{lem55-eqn3}-\eqref{prooftheo4-eq30}, we have
$$
\begin{array}{lcl}
\|I_{p,q}^{(5)}\|_{\mathcal{X}_{a,\sigma}^0}
&\leq&
\displaystyle \alpha \int_{\mathbb{R}^3}\int_{t_p}^{t_q}
e^{-(t_q-z)|\xi|^2}
e^{a|\xi|^{\frac{1}{\sigma}}}
\left|\mathcal{F}\left(\sum_{k=1}^3u_k^{2m+1}e_k\right)(z,\xi)\right|
\,dz\,d\xi
\\[2mm]
&\leq&
\displaystyle \alpha\int_{t_p}^{t_q}
\int_{\mathbb{R}^3}
e^{a|\xi|^{\frac{1}{\sigma}}}
\left|\mathcal{F}\left(\sum_{k=1}^3u_k^{2m+1}e_k\right)(z,\xi)\right|
\,d\xi\,dz
\\[2mm]
&\leq&
\displaystyle \alpha\int_{t_p}^{t_q}
\|u(z)\|_{\mathcal{X}_{a,\sigma}^0}^{2m+1}\,dz
\\[2mm]
&\leq&
\displaystyle \alpha M_*^{2m+1}(t_q-t_p).
\end{array}
$$
Since $t_q-t_p\to 0^+$ as $p,q\to\infty$, we obtain
$$
\lim_{p,q\to\infty}
\|I_{p,q}^{(5)}\|_{\mathcal{X}_{a,\sigma}^0}=0.
$$
\end{enumerate}
Finally, by the previous five estimates, $(u(t))_{t<T^*}$ is a Cauchy family in the Banach
space $\mathcal{X}_{a,\sigma}^0(\mathbb{R}^3)$. Hence, there exists a unique
$u^*\in\mathcal{X}_{a,\sigma}^0(\mathbb{R}^3)$ such that
$$
\lim_{t\to T^*}u(t)=u^*
\quad\text{in }\mathcal{X}_{a,\sigma}^0(\mathbb{R}^3).
$$
Now, consider $(NSE_M)$ with initial datum $ v(0)=u^*$. By Theorem~\ref{theo3}, there exists
$T_0=T_0(u^*)>0$ and a unique solution
$$
v\in C([0,T_0];\mathcal{X}_{a,\sigma}^0(\mathbb{R}^3))
\cap L^1([0,T_0];\mathcal{X}_{a,\sigma}^2(\mathbb{R}^3)).
$$
Define
$$
w(t)=
\begin{cases}
u(t), & \text{if } t\in[0,T^*),\\
v(t-T^*), & \text{if } t\in[T^*,T^*+T_0).
\end{cases}
$$
Then $w$ is a solution of $(NSE_M)$ on $[0,T^*+T_0)$, which extends the maximal solution $u$ beyond $T^*$. This contradiction shows that $u$ cannot remain bounded in $\mathcal{X}_{a,\sigma}^0(\mathbb{R}^3)$. Therefore,
$$
\limsup_{t\to T^*}
\|u(t)\|_{\mathcal{X}_{a,\sigma}^0}
=+\infty.
$$
\subsection{Proof of \eqref{theo4-eqn3}}\pn
From inequality \eqref{prooftheo4-eq33}, we have
\begin{eqnarray}\nonumber
&&\|u(t)\|_{\mathcal{X}_{a,\sigma}^0}
+\int_0^t
\|\Delta u(\tau)\|_{\mathcal{X}_{a,\sigma}^0}\,d\tau
\\ \label{prooftheo4-eq43}
&&\quad \leq
\|u^0\|_{\mathcal{X}_{a,\sigma}^0}
+\int_0^t
\|u\cdot\nabla u\|_{\mathcal{X}_{a,\sigma}^0}\,d\tau
+\alpha\int_0^t
\left\|\sum_{j=1}^3u_j^{2m+1}e_j\right\|_{\mathcal{X}_{a,\sigma}^0}\,d\tau .
\end{eqnarray}
Using Lemma~\ref{lem3} with $s=0$, together with \eqref{lem1-eqn2}, we obtain
$$
\begin{array}{lcl}
\|u\cdot\nabla u\|_{\mathcal{X}_{a,\sigma}^0}
&\leq&
4\|u\|_{\mathcal{X}_{\frac{a}{\sigma},\sigma}^0}
\|u\|_{\mathcal{X}_{a,\sigma}^1}
\\[1mm]
&\leq&
4\|u\|_{\mathcal{X}_{\frac{a}{\sigma},\sigma}^0}
\|u\|_{\mathcal{X}_{a,\sigma}^0}^{\frac12}
\|u\|_{\mathcal{X}_{a,\sigma}^2}^{\frac12}.
\end{array}
$$
Moreover, by \eqref{lem33-eqn1}, we have
$$
\left\|\sum_{j=1}^3u_j^{2m+1}e_j\right\|_{\mathcal{X}_{a,\sigma}^0}
\leq
2^{2m}
\|u\|_{\mathcal{X}_{a,\sigma}^0}
\|u\|_{\mathcal{X}_{\frac{a}{\sigma},\sigma}^0}^{2m}.
$$
Therefore, \eqref{prooftheo4-eq43} gives
$$
\begin{array}{lcl}
&&\|u(t)\|_{\mathcal{X}_{a,\sigma}^0}
+\displaystyle\int_0^t
\|\Delta u(\tau)\|_{\mathcal{X}_{a,\sigma}^0}\,d\tau
\\[2mm]
&&\quad \leq
\|u^0\|_{\mathcal{X}_{a,\sigma}^0}
+
4\displaystyle\int_0^t
\|u\|_{\mathcal{X}_{\frac{a}{\sigma},\sigma}^0}
\|u\|_{\mathcal{X}_{a,\sigma}^0}^{\frac12}
\|u\|_{\mathcal{X}_{a,\sigma}^2}^{\frac12}\,d\tau
\\[2mm]
&&\qquad
+
2^{2m}\alpha
\displaystyle\int_0^t
\|u\|_{\mathcal{X}_{a,\sigma}^0}
\|u\|_{\mathcal{X}_{\frac{a}{\sigma},\sigma}^0}^{2m}\,d\tau .
\end{array}
$$
Applying Young's inequality $xy\leq \frac{x^2}{2}+\frac{y^2}{2}$, we obtain
$$
\begin{array}{lcl}
&&\|u(t)\|_{\mathcal{X}_{a,\sigma}^0}
+\displaystyle\frac12\int_0^t
\|\Delta u(\tau)\|_{\mathcal{X}_{a,\sigma}^0}\,d\tau
\\[2mm]
&&\quad \leq
\|u^0\|_{\mathcal{X}_{a,\sigma}^0}
+
8\displaystyle\int_0^t
\|u\|_{\mathcal{X}_{\frac{a}{\sigma},\sigma}^0}^{2}
\|u\|_{\mathcal{X}_{a,\sigma}^0}\,d\tau
\\[2mm]
&&\qquad
+
2^{2m}\alpha
\displaystyle\int_0^t
\|u\|_{\mathcal{X}_{a,\sigma}^0}
\|u\|_{\mathcal{X}_{\frac{a}{\sigma},\sigma}^0}^{2m}\,d\tau
\\[2mm]
&&\quad \leq
\|u^0\|_{\mathcal{X}_{a,\sigma}^0}
+
C_{m,\alpha}
\displaystyle\int_0^t
\|u\|_{\mathcal{X}_{a,\sigma}^0}
\left(
\|u\|_{\mathcal{X}_{\frac{a}{\sigma},\sigma}^0}^{2}
+
\|u\|_{\mathcal{X}_{\frac{a}{\sigma},\sigma}^0}^{2m}
\right)d\tau .
\end{array}
$$
By Gronwall's lemma, we deduce that
$$
\|u(t)\|_{\mathcal{X}_{a,\sigma}^0}
\leq
\|u^0\|_{\mathcal{X}_{a,\sigma}^0}
\exp\left(
C_{m,\alpha}
\int_0^t
\left(
\|u(\tau)\|_{\mathcal{X}_{\frac{a}{\sigma},\sigma}^0}^{2}
+
\|u(\tau)\|_{\mathcal{X}_{\frac{a}{\sigma},\sigma}^0}^{2m}
\right)d\tau
\right).
$$
Assume, by contradiction, that
$$
\int_0^{T^*}
\|u(\tau)\|_{\mathcal{X}_{\frac{a}{\sigma},\sigma}^0}^{2m}\,d\tau
<\infty.
$$
Since $m\geq 1$, we have
$$
\|u(\tau)\|_{\mathcal{X}_{\frac{a}{\sigma},\sigma}^0}^{2}
\leq
1+
\|u(\tau)\|_{\mathcal{X}_{\frac{a}{\sigma},\sigma}^0}^{2m}.
$$
Hence,
$$
\int_0^{T^*}
\left(
\|u(\tau)\|_{\mathcal{X}_{\frac{a}{\sigma},\sigma}^0}^{2}
+
\|u(\tau)\|_{\mathcal{X}_{\frac{a}{\sigma},\sigma}^0}^{2m}
\right)d\tau
<\infty.
$$
Consequently,
$$
\sup_{0\leq t<T^*}
\|u(t)\|_{\mathcal{X}_{a,\sigma}^0}
<\infty,
$$
which contradicts \eqref{theo4-eqn2}. Therefore,
$$
\int_0^{T^*}
\|u(\tau)\|_{\mathcal{X}_{\frac{a}{\sigma},\sigma}^0}^{2m}\,d\tau
=+\infty.
$$
This proves \eqref{theo4-eqn3}.
\subsection{Proof of \eqref{theo4-eqn4}}\pn
From \eqref{theo4-eqn3}, we have
\begin{equation}\label{prooftheo4-eqn51}
\int_0^{T^*}
\|u(\tau)\|_{\mathcal{X}_{\frac{a}{\sigma},\sigma}^0}^{2m}\,d\tau
=+\infty.
\end{equation}
In particular,
$$
\limsup_{t\nearrow T^*}
\|u(t)\|_{\mathcal{X}_{\frac{a}{\sigma},\sigma}^0}
=+\infty.
$$
Let $a'=\frac{a}{\sigma}\in(0,a)$. Since
$$
\mathcal{X}_{a,\sigma}^0
\hookrightarrow
\mathcal{X}_{a',\sigma}^0,
$$
we have
$$
u\in C\big([0,T^*),\mathcal{X}_{a',\sigma}^0(\mathbb{R}^3)\big).
$$
Repeating the estimates of the previous subsection with $a'$ instead of $a$, we obtain
$$
\|u(t)\|_{\mathcal{X}_{a',\sigma}^0}
\leq
\|u^0\|_{\mathcal{X}_{a',\sigma}^0}
\exp\left(
C_{m,\alpha}T^*
+
C_{m,\alpha}
\int_0^t
\|u(\tau)\|_{\mathcal{X}_{\frac{a'}{\sigma},\sigma}^0}^{2m}
\,d\tau
\right).
$$
Since
$$
\mathcal{X}_{a',\sigma}^0
\hookrightarrow
\mathcal{X}_{\frac{a'}{\sigma},\sigma}^0,
$$
we get
$$
\|u(\tau)\|_{\mathcal{X}_{\frac{a'}{\sigma},\sigma}^0}
\leq
\|u(\tau)\|_{\mathcal{X}_{a',\sigma}^0}.
$$
Hence,
$$
\|u(t)\|_{\mathcal{X}_{a',\sigma}^0}
\leq
\|u^0\|_{\mathcal{X}_{a',\sigma}^0}
\exp\left(
C_{m,\alpha}T^*
+
C_{m,\alpha}
\int_0^t
\|u(\tau)\|_{\mathcal{X}_{a',\sigma}^0}^{2m}
\,d\tau
\right).
$$
Consequently,
$$
\|u(t)\|_{\mathcal{X}_{a',\sigma}^0}^{2m}
\leq
\|u^0\|_{\mathcal{X}_{a',\sigma}^0}^{2m}
\exp\left(
2mC_{m,\alpha}T^*
+
2mC_{m,\alpha}
\int_0^t
\|u(\tau)\|_{\mathcal{X}_{a',\sigma}^0}^{2m}
\,d\tau
\right)
$$
where $C_{m,\alpha}>0$ denotes a generic constant depending only on $m$ and $\alpha$, whose value may change from line to line.
Set
$$
Y(t)=\int_0^t
\|u(\tau)\|_{\mathcal{X}_{a',\sigma}^0}^{2m}\,d\tau.
$$
Then
$$
Y'(t)
\leq
\|u^0\|_{\mathcal{X}_{a',\sigma}^0}^{2m}
e^{2mC_{m,\alpha}T^*}
e^{2mC_{m,\alpha}Y(t)}.
$$
Equivalently,
$$
Y'(t)e^{-2mC_{m,\alpha}Y(t)}
\leq
\|u^0\|_{\mathcal{X}_{a',\sigma}^0}^{2m}
e^{2mC_{m,\alpha}T^*}.
$$
Integrating over \([0,t]\), we obtain
$$
\frac{1-e^{-2mC_{m,\alpha}Y(t)}}{2mC_{m,\alpha}}
\leq
t\,
\|u^0\|_{\mathcal{X}_{a',\sigma}^0}^{2m}
e^{2mC_{m,\alpha}T^*}.
$$
Letting $t\nearrow T^*$ and using \eqref{prooftheo4-eqn51}, we deduce
$$
1
\leq
2mC_{m,\alpha}T^*
e^{2mC_{m,\alpha}T^*}
\|u^0\|_{\mathcal{X}_{a',\sigma}^0}^{2m}.
$$
Let $r_0\in[0,T^*)$ and consider the following Navier-Stokes equations with damping
$$
(S_{r_0})
\left\{\begin{array}{l}
\partial_{t} v- \Delta v+v.\nabla v +\alpha \sum_{k=1}^3v_k^{2m+1}e_k =-\nabla q \text { in } \mathbb{R}^{+}
\times \mathbb{R}^{3} \\
\operatorname{div} v=0 \text { in } \mathbb{R}^{+} \times \mathbb{R}^{3} \\
v(0, x)=u(r_0,x) \text { in } \mathbb{R}^{3}.
\end{array}\right.
$$
By the last results, $(S_{r_0})$ admits a unique maximal solution $v\in C([0,S^*),\mathcal{X}^0_{a,\sigma}(\mathbb{R}^3))$, with $0<S^*\leq \infty$. Moreover, if $S^*$ is finite, then
\begin{equation}\label{eqfn}
1
\leq
2mC_{m,\alpha}S^*
e^{2mC_{m,\alpha}S^*}
\|u(r_0)\|_{\mathcal{X}_{a',\sigma}^0}^{2m}.
\end{equation}
By existence and uniqueness of maximal solution to the system $(S_{r_0})$, we get $S^*=T^*-r_0$ and
$$v(t)=u(r_0+t),\;\forall t\in[0,T^*-r_0).$$
Moreover by using (\ref{eqfn}), we get
$$
1
\leq
2mC_{m,\alpha}(T^*-r_0)
e^{2mC_{m,\alpha}(T^*-r_0)}
\|u(r_0)\|_{\mathcal{X}_{a',\sigma}^0}^{2m},
$$
then
$$
1\leq 2mC_{m,\alpha}(T^*-r_0)
e^{2mC_{m,\alpha}T^*}
\|u(r_0)\|_{\mathcal{X}_{a',\sigma}^0}^{2m}.
$$
Therefore,
$$
\|u(r_0)\|_{\mathcal{X}_{\frac{a}{\sigma},\sigma}^0}
\geq
C\,(T^*-r_0)^{-\frac{1}{2m}},
$$
where
$$
C=
\left(
2mC_{m,\alpha}
e^{2mC_{m,\alpha}T^*}
\right)^{-\frac{1}{2m}}>0.
$$
Since $r_0\in[0,T^*)$ is arbitrary, we conclude that
$$
\|u(t)\|_{\mathcal{X}_{\frac{a}{\sigma},\sigma}^0}
\geq
C\,(T^*-t)^{-\frac{1}{2m}},
\qquad 0\leq t<T^*.
$$
This proves \eqref{theo4-eqn4}.

\subsection{Proof of (\ref{theo4-eqn5})}

Since
$$
\mathcal{X}_{a,\sigma}^0(\mathbb{R}^3)
\hookrightarrow
\mathcal{X}_{\frac{a}{\sigma},\sigma}^0(\mathbb{R}^3),
$$
we have
$$
u\in
C\left([0,T^*),
\mathcal{X}_{\frac{a}{\sigma},\sigma}^0(\mathbb{R}^3)\right).
$$
Moreover, by \eqref{theo4-eqn4},
$$
\|u(t)\|_{\mathcal{X}_{\frac{a}{\sigma},\sigma}^0}
\geq
C(T^*-t)^{-\frac1{2m}},
$$
and hence
$$
\lim_{t\nearrow T^*}
\|u(t)\|_{\mathcal{X}_{\frac{a}{\sigma},\sigma}^0}
=+\infty.
$$
Therefore, $T^*$ is also the maximal existence time of $u$ in
$\mathcal{X}_{\frac{a}{\sigma},\sigma}^0(\mathbb{R}^3)$.

We may therefore apply \eqref{theo4-eqn4} again, with $a$ replaced by
$a/\sigma$. This yields the same lower bound in
$\mathcal{X}_{\frac{a}{\sigma^2},\sigma}^0$ and shows that $T^*$ is also
the maximal existence time in this space. Repeating this argument
inductively, we obtain, for every $n\in\mathbb{N}$,
\begin{equation}\label{eqfn2}
1\leq
2mC_{m,\alpha}(T^*-t)
e^{2mC_{m,\alpha}T^*}
\|u(t)\|_{\mathcal{X}_{\frac{a}{\sigma^n},\sigma}^0}^{2m},
\qquad 0\leq t<T^*.
\end{equation}
Now, fix $t\in[0,T^*)$. Since $\sigma>1$, we have
$$
\frac{a}{\sigma^n}\longrightarrow0
\qquad\text{as }n\to\infty,
$$
and hence
$$
e^{\frac{a}{\sigma^n}|\xi|^{1/\sigma}}
\longrightarrow1.
$$
Moreover,
$$
e^{\frac{a}{\sigma^n}|\xi|^{1/\sigma}}
|\widehat{u}(t,\xi)|
\leq
e^{a|\xi|^{1/\sigma}}
|\widehat{u}(t,\xi)|,
$$
and the function on the right-hand side belongs to
$L^1(\mathbb{R}^3)$. Therefore, by the dominated convergence theorem,
$$
\lim_{n\to\infty}
\|u(t)\|_{\mathcal{X}_{\frac{a}{\sigma^n},\sigma}^0}
=
\|u(t)\|_{\mathcal{X}^0}.
$$
Passing to the limit as $n\to\infty$ in \eqref{eqfn2}, we obtain
$$
1\leq
2mC_{m,\alpha}(T^*-t)
e^{2mC_{m,\alpha}T^*}
\|u(t)\|_{\mathcal{X}^0}^{2m}.
$$
Consequently,
$$
\|u(t)\|_{\mathcal{X}^0}
\geq
\left(
2mC_{m,\alpha}e^{2mC_{m,\alpha}T^*}
\right)^{-\frac1{2m}}
(T^*-t)^{-\frac1{2m}}.
$$
In particular,
$$
\lim_{t\nearrow T^*}\|u(t)\|_{\mathcal{X}^0}=+\infty.
$$
Hence, $T^*$ is also the maximal existence time in
$\mathcal{X}^0(\mathbb{R}^3)$, and \eqref{theo4-eqn5} follows.

 \end{document}